\documentclass[11pt,twoside,a4paper]{amsart}

\usepackage{amssymb}
\usepackage{vmargin}
\usepackage{amscd}
\usepackage{stmaryrd}
\usepackage{mathrsfs}
\usepackage[all]{xy}
\usepackage{xr}
\usepackage{hyperref}

\setmargins{32mm}{20mm}{14.6cm}{22cm}{1cm}{1cm}{1cm}{1cm}

\newcommand\la{\leftarrow}
\newcommand\lra{\longrightarrow}

\newcommand\ten{\otimes}

\newcommand\eps{\epsilon}

\renewcommand\H{\mathrm{H}}
\newcommand\z{\mathrm{Z}}

\newcommand\Z{\mathbb{Z}}
\newcommand\Q{\mathbb{Q}}

\newcommand\bD{\mathbb{D}}

\newcommand\bG{\mathbb{G}}

\newcommand\bL{\mathbb{L}}

\newcommand\C{\mathcal{C}}

\newcommand\cA{\mathcal{A}}

\newcommand\cP{\mathcal{P}}

\renewcommand\O{\mathscr{O}}

\newcommand\fX{\mathfrak{X}}

\newcommand\g{\mathfrak{g}}

\newcommand\cHom{\mathcal{H}\!\mathit{om}}

\newcommand\Alg{\mathrm{Alg}}
\newcommand\CAlg{\mathrm{CAlg}}

\newcommand\Mod{\mathrm{Mod}}

\newcommand\Hom{\mathrm{Hom}}

\newcommand\map{\mathrm{map}}

\newcommand\HHom{\underline{\mathrm{Hom}}}

\newcommand\cone{\mathrm{cone}}
\newcommand\cocone{\mathrm{cocone}}

\newcommand\ob{\mathrm{ob}}

\newcommand\Co{\mathrm{Co}}
\newcommand\CoS{\mathrm{CoS}}

\newcommand\Spec{\mathrm{Spec}\,}

\newcommand\Set{\mathrm{Set}}

\newcommand\Aff{\mathrm{Aff}}

\newcommand\Sp{\mathrm{Sp}}
\newcommand\PreSp{\mathrm{PreSp}}

\newcommand\Pol{\mathrm{Pol}}

\newcommand\Com{\mathrm{Com}}
\newcommand\Comp{\mathrm{Comp}}

\newcommand\nondeg{\mathrm{nondeg}}

\newcommand\Lim{\varprojlim}
\newcommand\LLim{\varinjlim}
\DeclareMathOperator*{\holim}{holim}
\newcommand\ho{\mathrm{ho}\!}
\newcommand\into{\hookrightarrow}
\newcommand\onto{\twoheadrightarrow}

\newcommand\xra{\xrightarrow}

\newcommand\pr{\mathrm{pr}}

\newcommand\bt{\bullet}
\newcommand\by{\times}

\newcommand\mc{\mathrm{MC}}
\newcommand\mmc{\underline{\mathrm{MC}}}

\newcommand\Symm{\mathrm{Symm}}

\newcommand\et{\acute{\mathrm{e}}\mathrm{t}}

\newcommand\cart{\mathrm{cart}}

\newcommand\Tot{\mathrm{Tot}\,}
\newcommand\diag{\mathrm{diag}\,}

\newcommand\pd{\partial}

\newcommand\half{\frac{1}{2}}

\newcommand\gr{\mathrm{gr}}

\newcommand\dR{\mathrm{dR}}
\newcommand\DR{\mathrm{DR}}

\newcommand\op{\mathrm{opp}}

\newcommand\co{\colon\thinspace}

\newcommand\oR{\mathbf{R}}

\newcommand\oL{\mathbf{L}}

\newcommand\uleft\underleftarrow
\newcommand\uline\underline
\newcommand\uright\underrightarrow

\newcommand{\tps}{\texorpdfstring}
\newtheorem{theorem}{Theorem}[section]
\newtheorem{proposition}[theorem]{Proposition}
\newtheorem{corollary}[theorem]{Corollary}

\newtheorem{lemma}[theorem]{Lemma}
\newtheorem*{theorem*}{Theorem}
\newtheorem*{proposition*}{Proposition}
\newtheorem*{corollary*}{Corollary}
\newtheorem*{lemma*}{Lemma}
\newtheorem*{conjecture*}{Conjecture}

\theoremstyle{definition}
\newtheorem{definition}[theorem]{Definition}
\newtheorem*{definition*}{Definition}

\newtheorem*{notation*}{Notation}

\theoremstyle{remark}
\newtheorem{example}[theorem]{Example}
\newtheorem{examples}[theorem]{Examples}
\newtheorem{remark}[theorem]{Remark}

\newtheorem{properties}[theorem]{Properties}

\newtheorem*{example*}{Example}
\newtheorem*{examples*}{Examples}
\newtheorem*{remark*}{Remark}
\newtheorem*{remarks*}{Remarks}
\newtheorem*{exercise*}{Exercise}
\newtheorem*{property*}{Property}
\newtheorem*{properties*}{Properties}

\begin{document}

\begin{abstract}
 We show that on a derived Artin $N$-stack, there is a canonical equivalence between the spaces of  $n$-shifted  symplectic structures and non-degenerate $n$-shifted Poisson structures.
\end{abstract}

\title{Shifted Poisson and symplectic structures on derived $N$-stacks}
\author{J.P.Pridham}
\thanks{This work was supported by  the Engineering and Physical Sciences Research Council [grant number EP/I004130/2].}

\maketitle

\section*{Introduction}

In \cite{PTVV}, the notion of an $n$-shifted symplectic structure for a derived Artin stack was introduced. The definition of an $n$-Poisson structure was also sketched, and an approach suggested to prove  equivalence  of $n$-shifted symplectic structures and  non-degenerate $n$-Poisson structures, by first  establishing deformation quantisation for symplectic structures.
 The Darboux theorems of \cite{BBBJdarboux,BouazizGrojnowski} imply the existence of shifted  Poisson structures locally on derived Deligne--Mumford stacks, but not globally.

We prove  existence of shifted  Poisson structures by a direct approach, not involving deformation quantisation  or Darboux theorems. 
The key new notion is that of compatibility between   an  $n$-shifted pre-symplectic (i.e. closed, possibly degenerate) structure $\omega$ and an $n$-Poisson structure $\pi$. In the case of unshifted structures on underived manifolds, compatibility simply says that the associated maps between tangent and cotangent bundles satisfy 
\[
 \pi^{\sharp} \circ \omega^{\sharp} \circ \pi^{\sharp}=\pi^{\sharp},
\]
which ensures that $\omega^{\sharp}=(\pi^{\sharp})^{-1}$ whenever $\pi$ is non-degenerate. We demonstrate that  this notion admits a natural generalisation incorporating shifts and higher coherence data. We expect that this notion of compatibility between Poisson and pre-symplectic structures is the same as  in Definition 1.4.14 of the contemporaneous treatment \cite{CPTVV}, whose proof also does not involve deformation quantisation. Our   notion of compatibility is extended in  \cite{DQvanish,DQnonneg,DQLag} from Poisson structures to quantisations. 
 
Our  first main observation  is that there is a canonical global section $\sigma$ of the tangent space $T\cP(X,n)$ of the space $\cP(X,n)$ of $n$-Poisson structures, given by differentiating the $\bG_m$-action on the differential graded Lie algebra of polyvectors. For an unshifted Poisson structure on a smooth underived scheme, this just maps $\pi$ to itself. In general, if we write $\pi = \sum_{r \ge 2} \pi_r$ with $\pi_r$ an $r$-vector, then
\[
 \sigma(\pi) = \sum_{r \ge 2} (r-1)\pi_r.
\]

For any  $n$-shifted  Poisson structure $\pi$, contraction gives  a multiplicative map $\mu(-,\pi)$ from de Rham to Poisson cohomology, so 
   any  $n$-shifted pre-symplectic structure $\omega$ defines another global section $\mu(\omega,-)$ of the tangent space of the space of Poisson structures. 
In the unshifted case on a smooth scheme, the associated map from the cotangent space to the tangent space is simply  given by
\[
 \mu(\omega, \pi)^{\sharp}= \pi^{\sharp} \circ \omega^{\sharp} \circ \pi^{\sharp}.
\]

We then say that a pair $(\omega, \pi)$ is compatible if $\mu(\omega,\pi) \simeq \sigma(\pi)$. More formally, the space of compatible pairs is the homotopy 
limit of the diagram 
\[
\xymatrix@1{ \PreSp(X,n) \by \cP(X,n) \ar@<0.5ex>[r]^-{\mu}  \ar@<-0.5ex>[r]_-{\sigma} & T\cP(X,n) \ar[r] & \cP(X,n)}.
\]
 In the unshifted case,  this amounts to seeking compatible pairs in the sense above. 

Poisson structures are functorial with respect to \'etale morphisms, so the notions above are readily defined for derived DM $N$-stacks. 
For derived  Artin $N$-stacks, the formulation of Poisson structures is more subtle:   we show that a derived Artin $N$-stack admits an \'etale cover by derived stacks coming from commutative bidifferential bigraded algebras, and put Poisson structures on them. These algebras perform the same role as the  formal affine derived stacks of \cite{CPTVV}, and we refer to them as stacky CDGAs.

We then show that: 
\begin{enumerate}
 \item If $\pi$ is non-degenerate, then every compatible pre-symplectic structure $\omega$ is also non-degenerate, hence symplectic, and the space of such structures is contractible.

\item If $\omega$ is symplectic, then the space of \emph{non-degenerate} compatible Poisson structures is contractible.
\end{enumerate}
Thus the spaces of $n$-shifted symplectic structures and non-degenerate $n$-Poisson structures are both weakly equivalent to the space of non-degenerate compatible pairs $(\omega, \pi)$. These results have recently been obtained in \cite{CPTVV}, using a very different method to formulate compatibility.

The structure of the paper is as follows.

Section \ref{affinesn} addresses the case of a single, fixed, derived affine scheme. The compatibility operator $\mu$ is introduced in Definition \ref{mudef}, and the key technical result is  Lemma \ref{keylemma}, which shows how $\mu$ interpolates between the de Rham differential and the Schouten--Nijenhuis bracket. \S \ref{towersn} shows how to realise the spaces of pre-symplectic structures, Poisson structures and compatible pairs as towers of homotopy fibres. This uses the obstruction theory associated to  
pro-nilpotent $L_{\infty}$-algebras, regarding the construction as a generalised deformation problem. The correspondence between symplectic and non-degenerate Poisson structures in the affine case is then given in  Corollary \ref{compatcor2}.

Section \ref{DMsn} then applies these results to derived stacks. 
To date, the only proof that the cotangent complex governs deformations of a  (derived)  algebraic stack (rather than just morphisms to it from a fixed object) uses simplicial resolutions by affine schemes (cf. \cite[Theorem 1.2]{aoki}, \cite[Theorem \ref{stacks2-deformstack}]{stacks2}). Since the construction of Poisson structures is  a kind of  generalised deformation problem, it is unsurprising that it can be   studied using such resolutions. For derived Deligne--Mumford $N$-stacks, we look at spaces of structures on the diagrams given by suitable simplicial resolutions, and show that they are independent of the resolution; the key being  \'etale functoriality. Theorem \ref{DMthm} then establishes an equivalence between the spaces of $n$-symplectic and non-degenerate $n$-Poisson structures. 

In section \ref{Artinsn},  these results are extended to derived Artin $N$-stacks. Given a simplicial resolution by derived affines, we can form an associated stacky CDGA, and indeed a simplicial resolution by stacky CDGAs (Corollary \ref{gooddescent}). The morphisms in this resolution resemble \'etale maps, allowing a suitably functorial generalisation of the results of \S \ref{affinesn} to stacky CDGAs.  The main result is then Theorem \ref{Artinthm}, which establishes an equivalence between the spaces of $n$-symplectic and non-degenerate $n$-Poisson structures on  derived Artin $N$-stacks. 

I would like to thank Victor Ginzburg for suggesting a simplification  to the proofs of Lemma \ref{keylemma} and Proposition \ref{compatP1}. I would also like to thank the anonymous referees for helpful suggestions and comments.

\tableofcontents

\section{Compatible structures on derived affines}\label{affinesn}

For the purposes of this section, we will fix a  graded-commutative algebra $R=R_{\bt}$ in chain complexes over $\Q$, and a cofibrant graded-commutative $R_{\bt}$-algebra $A=A_{\bt}$. We will denote the differential on $A$ by $\delta$. 


In particular, the cofibrancy condition holds whenever the underlying morphism $R_{\#} \to A_{\#}$ of graded commutative algebras is freely generated in non-negative chain degrees. It ensures that the module $\Omega^1_{A/R}$ of K\"ahler differentials is a model for the cotangent complex of $A$ over $R$. For the purposes of this section, this latter property is all we need, so we could relax the cofibrancy condition slightly to include morphisms which are ind-smooth rather than freely generated. 

We define $\Omega^p_{A/R}:= \Lambda_A^p \Omega^1_{A/R}$, and we  denote its differential (inherited from $A$) by $\delta$. There is also a de Rham cochain differential $\Omega^p_{A/R} \to \Omega^{p+1}_{A/R}$, which we denote by $d$.

For a chain (resp. cochain) complex $M$, we write $M_{[i]}$ (resp. $M^{[j]}$) for the complex $(M_{[i]})_m= M_{i+m}$ (resp. $(M^{[j]})^m = M^{j+m}$).
Given $A$-modules $M,N$ in chain complexes, we write $\HHom_A(M,N)$ for the cochain complex given by
\[
 \HHom_A(M,N)^i= \Hom_{A_{\#}}(M_{\#},N_{\#[-i]}),
\]
with differential $\delta f= \delta_N \circ f \pm f \circ \delta_M$,
where $V_{\#}$ denotes the graded vector space underlying a chain complex $V$.

\subsection{Shifted Poisson structures}\label{poisssn}

\subsubsection{Polyvector fields}\label{polsn}

\begin{definition}\label{poldef}
Define the cochain complex of $n$-shifted polyvector fields  on $A$ by
\[
 \widehat{\Pol}(A,n):= \HHom_A(\CoS_A((\Omega^1_{A/R})_{[-n-1]}),A), 
\]
with graded-commutative  multiplication following the usual conventions for symmetric powers.  [Here, $\CoS_A^p(M) =\Co\Symm^p_A(M)= (M^{\ten_A p})^{\Sigma_p}$ and $\CoS_A(M) = \bigoplus_{p \ge 0}\CoS_A^p(M)$.]

The Lie bracket on $\Hom_A(\Omega^1_{A/R},A)$ then extends to give a bracket (the Schouten--Nijenhuis bracket)
\[
[-,-] \co \widehat{\Pol}(A,n)\by \widehat{\Pol}(A,n)\to \widehat{\Pol}(A,n)^{[-1-n]},
\]
determined by the property that it is a bi-derivation with respect to the multiplication operation. 

Thus  $\widehat{\Pol}(A,n)$ has the natural structure of a $P_{n+2}$-algebra  (i.e. an $(n+1)$-shifted Poisson algebra), and in particular $\widehat{\Pol}(A,n)^{[n+1]}$ is a differential graded Lie algebra (DGLA) over $R$.

Note that the cochain differential $\delta$ on $\widehat{\Pol}(A,n)$ can be written as $[\delta,-]$, where $\delta \in \widehat{\Pol}(A,n)^{n+2}$ is the element defined by the derivation $\delta$ on $A$.
\end{definition}

Strictly speaking, $\widehat{\Pol}$ is the complex of multiderivations, as polyvectors are usually defined as symmetric powers of the tangent complex. The two definitions agree (modulo completion) whenever the tangent complex is perfect, and Definition \ref{poldef} is the more natural object when the definitions differ.

\begin{definition}\label{Fdef}
Define a decreasing filtration $F$ on  $\widehat{\Pol}(A,n)$ by 
\[
 F^i\widehat{\Pol}(A,n):= \HHom_A( \bigoplus_{j \ge i} \CoS_A^j((\Omega^1_{A/R})_{[-n-1]}),A);
\]
this has the properties that $\widehat{\Pol}(A,n)= \Lim_i \widehat{\Pol}(A,n)/F^i$, with $[F^i,F^j] \subset F^{i+j-1}$, $\delta F^i \subset F^i$, and $F^i F^j \subset F^{i+j}$.
\end{definition}

Observe that this filtration makes $F^2\widehat{\Pol}(A,n)^{[n+1]}$ into a pro-nilpotent DGLA.

\subsubsection{Poisson structures}

\begin{definition}\label{mcPLdef}
 Given a   DGLA $L$, define the the Maurer--Cartan set by 
\[
\mc(L):= \{\omega \in  L^{1}\ \,|\, \delta\omega + \half[\omega,\omega]=0 \in  \bigoplus_n L^{2}\}.
\]

Following \cite{hinstack}, define the Maurer--Cartan space $\mmc(L)$ (a simplicial set) of a nilpotent  DGLA $L$ by
\[
 \mmc(L)_n:= \mc(L\ten_{\Q} \Omega^{\bt}(\Delta^n)),
\]
where 
\[
\Omega^{\bt}(\Delta^n)=\Q[t_0, t_1, \ldots, t_n,\delta t_0, \delta t_1, \ldots, \delta t_n ]/(\sum t_i -1, \sum \delta t_i)
\]
is the commutative dg algebra of de Rham polynomial forms on the $n$-simplex, with the $t_i$ of degree $0$.
\end{definition}

\begin{definition}
Given an inverse system $L=\{L_{\alpha}\}_{\alpha}$ of nilpotent DGLAs, define
\[
 \mc(L):= \Lim_{\alpha} \mc(L_{\alpha}) \quad  \mmc(L):= \Lim_{\alpha} \mmc(L_{\alpha}).
\]
Note that  $\mc(L)= \mc(\Lim_{\alpha}L_{\alpha})$, but $\mmc(L)\ne \mmc(\Lim_{\alpha}L_{\alpha}) $. 
\end{definition}

\begin{definition}\label{poissdef}
Define an $R$-linear  $n$-shifted Poisson structure on $A$ to be an element of
\[
 \mc(F^2 \widehat{\Pol}(A,n)^{[n+1]}),
\]
 and the space $\cP(A,n)$ of $R$-linear  $n$-shifted Poisson structures on $A$ to be given by the simplicial 
set
\[
 \cP(A,n):= \mmc( \{F^2 \widehat{\Pol}(A,n)^{[n+1]}/F^{i+2}\}_i).
\]

Also write $\cP(A,n)/F^{i+2}:= \mmc(F^2 \widehat{\Pol}(A,n)^{[n+1]}/F^{i+2})$, so $\cP(A,n)= \Lim_i \cP(A,n)/F^{i+2}$.
\end{definition}


\begin{remark}
Observe that elements of $\cP_0(A,n)= \mc(F^2 \widehat{\Pol}(A,n)^{[n+1]})$ consist of infinite sums $\pi = \sum_{i \ge 2}\pi_i$ with 
\[
 \pi_i \co \CoS_A^i((\Omega^1_{A/R})_{[-n-1]}) \to A_{[-n-2]}
\]
 satisfying $\delta(\pi_i) + \half \sum_{j+k=i+1} [\pi_j,\pi_k]=0$. This is precisely the condition which ensures that $\pi$ defines an $L_{\infty}$-algebra structure on $A_{[-n]}$. This then makes $A$ into a $\hat{P}_{n+1}$-algebra in the sense of  \cite[Definition 2.9]{melaniPoisson}. Equivalently, this is an algebra for the operad $\Com \circ (L_{\infty}[-n])$ via a distributive law $  (L_{\infty}[-n])\circ\Com \to \Com \circ (L_{\infty}[-n])$.

 It is important to remember that  $ \cP(A,n)$ is just one explicit presentation of the right-derived functor of $\cP_0(A,n) $ for varying $R$.
\end{remark}

When we need to compare chain and cochain complexes, we  silently make use of the equivalence  $u$ from chain complexes to cochain complexes given by $(uV)^i := V_{-i}$. On suspensions, this has the effect that $u(V_{[n]}) = (uV)^{[-n]}$.

\begin{definition}
We say that an  $n$-shifted Poisson structure $\pi = \sum_{i \ge 2}\pi_i $ is non-degenerate if $\pi_2 \co \CoS_A^2((\Omega^1_{A/R})_{[-n-1]}) \to A_{[-n-2]}$ induces a quasi-isomorphism 
\[
\pi_2^{\sharp}\co  (\Omega^1_{A/R})_{[-n]} \to \HHom_A(\Omega^1_{A/R},A)
\]
and $\Omega^1_{A/R} $ is perfect. 

Define $\cP(A,n)^{\nondeg}\subset \cP(A,n)$ to consist of non-degenerate elements --- this is a union of path-components.
\end{definition}


\subsubsection{The canonical tangent vector $\sigma$}

The space $\cP(A,n) $ admits an action of $\bG_m(R_0)$, which is inherited from the scalar multiplication on $\widehat{\Pol}(A,n)$ in which $\Hom_A(\CoS_A^p((\Omega^1_{A/R})_{[-n-1]}),A)$ is given weight $p-1$. Differentiating this action gives us a global tangent vector on $\cP(A,n) $, as follows. Take $\eps$ to be a variable of degree $0$, with $\eps^2=0$.

\begin{definition}\label{Tpoissdef}
Define the tangent spaces
\begin{eqnarray*}
 T\cP(A,n)&:=& \mmc( \{F^2 \widehat{\Pol}(A,n)^{[n+1]}\ten_{\Q} \Q[\eps]/F^{i+2}\}_i)\\
T\cP(A,n)/F^{i+2}&:=& \mmc( F^2 \widehat{\Pol}(A,n)^{[n+1]}\ten_{\Q} \Q[\eps]/F^{i+2}).
\end{eqnarray*}
 \end{definition}
These are simplicial sets over $\cP(A,n)$ (resp. $\cP(A,n)/F^{i+2}$), fibred in simplicial abelian groups. 

\begin{definition}\label{cPdef}
Given $\pi \in  \cP_0(A,n)$, observe that $\delta+[\pi,-]$ defines a square-zero derivation on $\widehat{\Pol}(A,n) $, and denote the resulting complex by 
\[
 \widehat{\Pol}_{\pi}(A,n).
\]
The product and bracket on polyvectors make this a  $P_{n+2}$-algebra, and it inherits the filtration $F$. 

Given $\pi \in  \cP_0(A,n)/F^p$, we define $\widehat{\Pol}_{\pi}(A,n)/F^p$ similarly. This is a CDGA, and $ F^1\widehat{\Pol}_{\pi}(A,n)/F^p$ is a $P_{n+2}$-algebra, because $F^i\cdot F^j \subset F^{i+j}$ and $[F^i,F^j] \subset F^{i+j-1}$.
\end{definition}

Note that $\widehat{\Pol}_{\pi}(A,n)$ is just the natural $n$-shifted analogue of the complex computing Poisson cohomology.

The following is an instance of a  standard result on square-zero extensions of DGLAs:
\begin{lemma}
 The  fibre $T_{\pi}\cP(A,n)$ of $T\cP(A,n)$  over $\pi \in \cP(A,n)$ is canonically homotopy equivalent to the Dold--Kan denormalisation of the good truncation   $\tau^{\le 0} (F^2\widehat{\Pol}_{\pi}(A,n)^{[n+2]})$. In particular, its
 homotopy groups  are given by
\[
 \pi_iT_{\pi}\cP(A,n)= \H^{n+2-i}(F^2 \widehat{\Pol}(A,n), \delta +[\pi,-]).
\]
\end{lemma}
The corresponding statements for $T_{\pi}\cP(A,n)/F^{i+2}$ also hold.

Now observe that the map
\begin{align*}
 \sigma \co F^2 \widehat{\Pol}(A,n)^{[n+1]} &\to F^2 \widehat{\Pol}(A,n)^{[n+1]}\ten_{\Q} \Q[\eps]\\
 \sum_{i \ge 2} \alpha_i  &\mapsto \sum_{i \ge 2} (\alpha_i+ (i-1)\alpha_i\eps)
\end{align*}
is a morphism of filtered DGLAs, for $\alpha_i \co \CoS_A^i(\Omega^1_{A/R}[n+1]) \to A$. This can be seen either by direct calculation or by observing that $\sigma$ is the differential of the $\bG_m$-action on $\Pol$.

\begin{definition}\label{sigmadef}
 Define the canonical tangent vector $\sigma \co \cP(A,n) \to  T\cP(A,n)$ on the space of $n$-shifted Poisson structures by applying $\mmc$ to the morphism $\sigma$ of DGLAs.
\end{definition}

Explicitly, this sends $\pi= \sum \pi_i $ to $\sigma(\pi)=\sum_{i \ge 2} (i-1)\pi_i \in T_{\pi}\cP(A,n)$, which thus has the property that $\delta \sigma(\pi) +[\pi, \sigma(\pi)]=0$.

The map $\sigma$ preserves the cofiltration in the sense that it comes from a system of  maps $\sigma \co  \cP(A,n)/F^{i+2} \to  T\cP(A,n)/F^{i+2} $.

\subsection{Shifted pre-symplectic structures}\label{prespsn}

\begin{definition}\label{DRdef}
Define the de Rham complex $\DR(A/R)$ to be the product total cochain complex of the double complex
\[
 A \xra{d} \Omega^1_{A/R} \xra{d} \Omega^2_{A/R}\xra{d} \ldots,
\]
so the total differential is $d \pm \delta$.

We define the Hodge filtration $F$ on  $\DR(A/R)$ by setting $F^p\DR(A/R) \subset \DR(A/R)$ to consist of terms $\Omega^i_{A/R}$ with $i \ge p$.
\end{definition}

Properties of the product total complex ensure that a map $f \co A \to B$ induces a quasi-isomorphism $\DR(A/R) \to \DR(B/R)$ whenever the maps $\Omega^p_{A/R} \to \Omega^p_{B/R}  $ are quasi-isomorphisms, which will happen whenever $f$ is a weak equivalence between cofibrant $R$-algebras.

The complex $\DR(A/R)$ has the natural structure of a commutative DG algebra over $R$, filtered in the sense that $F^iF^j \subset F^{i+j}$. 

\begin{definition}\label{presymplecticdef}
Define an $n$-shifted pre-symplectic structure $\omega$ on $A/R$ to be an element
\[
 \omega \in \z^{n+2}F^2\DR(A/R).
\]
\end{definition}

Explicitly, this means that $\omega$ is given by an infinite sum $\omega = \sum_{i \ge 2} \omega_i$, with $\omega_i \in (\Omega^i_{A/R})_{i-n-2}$ and $d\omega_i = \delta \omega_{i+1}$.

\begin{definition}
 Define an $n$-shifted symplectic structure $\omega$ on $A/R$ to be an $n$-shifted pre-symplectic structure $\omega$ for which the component $\omega_2 \in \z^n\Omega^2_{A/R}$ induces a quasi-isomorphism
\[
 \omega_2^{\sharp} \co \Hom_A(\Omega^1_{A/R},A) \to (\Omega^1_{A/R})_{[-n]}.
\]
and $\Omega^1_{A/R} $ is perfect as an $A$-module. 
\end{definition}

Now, we can regard $F^2\DR(A/R)^{[n+1]}$  as a filtered  DGLA with trivial bracket. This has the property that $\mc( F^2\DR(A/R)^{[n+1]})= \z^{n+2}F^2\DR(A/R)$.   We therefore make the following definition:

\begin{definition}\label{PreSpdef}
 Define the space of $n$-shifted pre-symplectic structures on $A/R$ to be the simplicial set
\[
 \PreSp(A,n):= \mmc( \{F^2\DR(A/R)^{[n+1]}/F^{i+2}\}_i) 
\]

Also write $\PreSp(A,n)/F^{i+2}:= \mmc(F^2  \DR(A/R)^{[n+1]}/F^{i+2})$, so $ \PreSp(A,n)= \Lim_i \PreSp(A,n)/F^{i+2}$.

Set $\Sp(A,n) \subset \PreSp(A,n)$ to consist of the symplectic structures --- this is a union of path-components.
\end{definition}
Note that  $\PreSp(A,n)/F^{i+2}$  is canonically weakly  equivalent to the Dold--Kan denormalisation of the complex $\tau^{\le 0}(F^2\DR(A)^{[n+2]}/F^{i+2})$ (and similarly for the limit  $ \PreSp(A,n)$),   but the description in terms of $\mmc$ will simplify comparisons.
 
\subsection{Compatible pairs}\label{compsn}

We will now develop the notion of compatibility between a pre-symplectic structure and a Poisson structure. Analogous results for the  unshifted, underived $\C^{\infty}$ context can be found in   \cite[Proposition 6.4]{KosmannSchwarzbachMagriPN}, and for the $\Z/2$-graded $\C^{\infty}$ context in \cite{KhudaverdianVoronov}.  

\begin{definition}\label{mudef}
Given  $\pi \in (F^2\widehat{\Pol}(A,n)/F^{p})^{n+2}$, define 
\[
 \mu(-,\pi) \co \DR(A/R)/F^{p} \to \widehat{\Pol}(A,n)/F^{p}
\]
to be the  morphism of graded $A$-algebras given on generators $ \Omega^1_{A/R}$ by 
\[
 \mu(a df, \pi):= \pi \lrcorner (a df)= a[\pi,f]. 
\]

Given $b \in (F^2\widehat{\Pol}(A,n)/F^{p})$, we then  define 
\[
 \nu(-, \pi, b) \co \DR(A/R)/F^{p} \to \widehat{\Pol}(A,n)/F^{p}
\]
by setting $\mu(\omega,\pi +b\eps)= \mu(\omega,\pi)+ \nu(-, \pi, b)\eps$ for $\eps^2=0$. More explicitly, $\nu(-, \pi, b)$ is  the $A$-linear derivation with respect to the ring homomorphism $ \mu(-,\pi) $ given on generators $ \Omega^1_{A/R}$ by 
\[
 \nu(a df, \pi, b):= b \lrcorner (a df)= a[b,f]. 
\]
\end{definition}

To see that these are well-defined in the sense that they descend to the quotients by $F^{p}$, observe that because $\pi \in F^2$, contraction has the property that
\[
 \mu(\Omega^1, \pi) \subset  F^1, \quad \nu(\Omega^1, \pi, b) \subset F^1,
\]
it follows that $\mu(F^p, \pi)\subset F^p$, $\nu(F^p,\pi,b)   \subset F^p$ by multiplicativity. 

Explicitly, when $\phi = a df_1 \wedge \ldots \wedge d f_p$, the operations are given by
\begin{align*}
\mu(\phi, \pi) &=a[\pi,f_1]\ldots [\pi,f_p,],\\
\nu(\phi, \pi, b)&= \sum_i \pm a[\pi,f_1]\ldots [b,f_i] \ldots [\pi,f_p].
\end{align*} 

\begin{lemma}\label{keylemma}
 For $\omega \in \DR(A)/F^p$ and $\pi \in F^2\widehat{\Pol}(A,n)^{n+2}/F^p$, we have
\begin{align*}
[\pi,\mu(\omega, \pi)] = \mu(d\omega, \pi) + \half \nu(\omega, \pi, [\pi,\pi]),\\
\delta_{\pi}\mu(\omega, \pi) = \mu(D\omega, \pi) + \nu(\omega, \pi,\kappa(\pi)),
\end{align*}
where $\delta_{\pi}= [\delta + \pi,-]$ is the differential on  $T_{\pi}\cP(A,n)/F^{p}$, with  $D= d \pm \delta$  the total differential on $(F^2\DR(A)/F^p) $ and $ \kappa(\pi)=[\delta, \pi] + \half [\pi,\pi]$.
\end{lemma}
\begin{proof}
 For fixed $\pi$ and varying $\omega$, all expressions are derivations with respect to $\mu(-, \pi)$,
 so it suffices to check this expression in the cases $p=0$ and  $p=1$, $\omega = df$. In these cases, we have
\begin{align*}
[\pi, \mu(a,\pi)] = [\pi, a] &= \mu(da, \pi),\\
\delta_{\pi}\mu(a, \pi)  &= \mu(D a, \pi),\\
 [\pi,\mu(df, \pi)]= [\pi, [\pi,f]]&= \half \nu(df, \pi, [\pi,\pi]),\\
\delta_{\pi}\mu(df, \pi)= \delta_{\pi}[\pi,f]&= \nu(df, \pi, [\delta, \pi] + \half [\pi,\pi]).
\end{align*}
Because $\nu(a, \pi, [\pi,\pi])=0 $ ($\nu$ being $A$-linear) 
 and $ddf=0$, this gives the required results.
\end{proof}

In particular, this implies that when $\pi$ is Poisson, $\mu(-,\pi)$ defines a map from de Rham cohomology to Poisson cohomology.

\begin{lemma}\label{mulemma}
There  are  maps
\[
 (\pr_2 + \mu\eps) \co  \PreSp(A,n)/F^{p} \by \cP(A,n)/F^{p} \to T\cP(A,n)/F^{p}
\]
 over $\cP(A,n)/F^{p}$ for all $p$, compatible with each other.
In particular, we have
\[
 (\pr_2 + \mu\eps)\co  \PreSp(A,n) \by \cP(A,n) \to T\cP(A,n).
\]
\end{lemma}
\begin{proof}
 For $\omega \in \PreSp(A,n)_0$, $\pi \in \cP(A,n)_0$,  Lemma \ref{keylemma} shows that $\mu(\omega, \pi) \in T_{\pi}\cP(A,n)/F^{p}$. 
Replacing $A,R$ with $A\ten \Omega^{\bt}(\Delta^m), R\ten \Omega^{\bt}(\Delta^m)$ then shows that the statement also holds on the $m$th 
level of the simplicial set.
\end{proof}

An alternative approach to proving Lemma \ref{mulemma} is to observe that $\mu$ defines a filtered $L_{\infty}$-morphism
\[
 F^2  \DR(A/R)^{[n+1]} \by  F^2 \widehat{\Pol}(A,n)^{[n+1]} \xra{\pr_2 + \mu \eps}  F^2 \widehat{\Pol}(A,n)^{[n+1]}[\eps]
\]
(with respect to the filtration $F$), and then to apply the functor $\mmc$.

\begin{example}\label{compatex1}
If $\omega$ and $\pi$ are pre-symplectic and Poisson, with $\omega_i=0$ for $i >2$ and $\pi_i=0$ for $i>2$,  then observe that 
$
 \mu(\omega, \pi)
$
 induces the morphism
\[
 \mu(\omega, \pi)^{\sharp} \co (\Omega^1_{A/R})_{[-n]}\to \Hom_A(\Omega^1_{A/R},A)
\]
given by 
\[
 \mu(\omega, \pi)^{\sharp} = \pi^{\sharp} \circ \omega^{\sharp} \circ \pi^{\sharp}.
\]
\end{example}

\begin{definition}
We say that an $n$-shifted  pre-symplectic structure $\omega$ and an $n$-Poisson structure $\pi$ are  compatible (or a compatible pair) if 
\[
 [\mu(\omega, \pi)] = [\sigma(\pi)] \in  \H^{n+2}(F^2\widehat{\Pol}_{\pi}(A,n)) =\pi_0T_{\pi}\cP(A,n),
\]
where $\sigma$ is the canonical tangent vector of Definition \ref{sigmadef}. 
\end{definition}

\begin{example}\label{compatex2}
Following Example \ref{compatex1}, if $\omega_i=0$ for $i >2$ and $\pi_i=0$ for $i>2$, then $(\omega, \pi)$ are a compatible pair
if and only if the map
\[
 \pi^{\sharp} \circ \omega^{\sharp} \circ \pi^{\sharp} \co (\Omega^1_{A/R})_{[-n]}\to \Hom_A(\Omega^1_{A/R},A)
\]
is homotopic to $\pi^{\sharp}$, because $\sigma(\pi)=\pi$ in this case. 

In particular, if $\pi$ is non-degenerate, this means that $\omega$ and $\pi$ determine each other up to homotopy. 
\end{example}

\begin{lemma}\label{compatnondeg}
If  $(\omega, \pi)$ is a compatible pair and $\pi$ is non-degenerate, then $\omega$ is symplectic.
\end{lemma}
\begin{proof}
Even when the vanishing conditions of Example \ref{compatex1} are not satisfied, we still have
\[
 \pi^{\sharp}_2 \circ \omega^{\sharp}_2 \circ \pi^{\sharp}_2 \simeq \pi^{\sharp}_2,
\]
so if $\pi^{\sharp}_2$ is a quasi-isomorphism, then  $\omega^{\sharp}_2$ must be its homotopy inverse.
\end{proof}

\begin{definition}\label{vanishingdef}
Given a simplicial set $Z$, an abelian group object $A$ in simplicial sets over $Z$, and a section $s \co Z \to A$, define the homotopy vanishing locus of $s$ to be the homotopy limit of the diagram
\[
\xymatrix@1{ Z \ar@<0.5ex>[r]^-{s}  \ar@<-0.5ex>[r]_-{0} & A \ar[r] & Z}.
\]
\end{definition}

We can write this as a homotopy fibre product $Z \by_{(s,0), A \by^h_Z A}^hA$, for the diagonal map $A \to A \by^h_Z A$. When $A$ is a trivial bundle $A = Z \by V$, for $V$ a simplicial abelian group, note that the homotopy vanishing locus is just the homotopy fibre of $s \co Z \to V$ over $0$. 

\begin{definition}\label{compdef}
Define the space $\Comp(A,n)$ of compatible $n$-shifted pairs to be the homotopy vanishing locus of
\[
 \mu - \sigma \co \PreSp(A,n) \by \cP(A,n) \to \pr_2^*T\cP(A,n) =\PreSp(A,n) \by T\cP(A,n).
\]

We define a cofiltration on this space by setting $ \Comp(A,n)/F^{p}$ to be the homotopy
vanishing locus of
\[
 \mu - \sigma \co \PreSp(A,n)/F^{p} \by \cP(A,n)/F^{p} \to \pr_2^*T\cP(A,n)/F^{p}.
\]
\end{definition}

We can rewrite $\Comp(A,n)$ as the homotopy limit of the diagram
\[
\xymatrix@1{  \PreSp(A,n) \by \cP(A,n) \ar@<0.5ex>[r]^-{(\pr_2 + \mu\eps)}  \ar@<-0.5ex>[r]_-{ \pr_2 +\sigma\pr_2\eps} & T\cP(A,n) \ar[r] & \cP(A,n) }
\]
of simplicial sets. 

In particular, an object of this space is given by a pre-symplectic structure $\omega$, a Poisson structure $\pi$, and a homotopy $h$ between $\mu(\omega,\pi)$ and $\sigma(\pi)$ in $T_{\pi}\cP(A,n)$.
%

\begin{definition}
 Define $\Comp(A,n)^{\nondeg} \subset \Comp(A,n)$ to consist of compatible pairs with $\pi$ non-degenerate. This is a union of path-components, and by Lemma \ref{compatnondeg} has a natural projection 
\[
 \Comp(A,n)^{\nondeg}\to \Sp(A,n)
\]
as well as the canonical map
\[
 \Comp(A,n)^{\nondeg} \to\cP(A,n)^{\nondeg}.
\]
\end{definition}

\begin{proposition}\label{compatP1}
The canonical map
\begin{eqnarray*}
    \Comp(A,n)^{\nondeg} \to  \cP(A,n)^{\nondeg}           
\end{eqnarray*}
 is a weak equivalence.
\end{proposition}
\begin{proof}
For any $\pi \in \cP(A,n)$, the homotopy fibre of $\Comp(A,n)^{\nondeg} $ over $\pi$ is just the homotopy fibre of
\[
\mu(-,\pi)  \co \PreSp(A,n)  \to T_{\pi}\cP(A,n) 
\]
over $\sigma(\pi)$.

The map $\mu(-,\pi) \co \DR(A/R) \to (\widehat{\Pol}(A,n), \delta_{\pi})$ is a morphism of complete filtered CDGAs by Lemma \ref{keylemma}, and non-degeneracy of $\pi_2$ implies that we have a quasi-isomorphism on the associated gradeds $\gr_F$.  We therefore have a quasi-isomorphism of filtered complexes, so we have isomorphisms on homotopy groups:
\begin{eqnarray*}
 \pi_j\PreSp(A,n)  &\to& \pi_jT_{\pi}\cP(A,n)\\
 \H^{n+2-j}(F^2 \DR(A/R)) &\to&  \H^{n+2-j}(F^2\widehat{\Pol}(A,n), \delta_{\pi}).
\end{eqnarray*}
\end{proof}

For an earlier analogue of this result, see \cite{KhudaverdianVoronov}, which takes $\delta=0$ and works in the $\Z/2$-graded (rather than $\Z$-graded) context, allowing $\pi=\sum_{i\ge 0} \pi_i$ to have constant and linear terms. It that setting, they show that for $\pi$ non-degenerate, there is a unique solution $\omega$ of the equation $\mu(\omega, \pi)  = \sigma(\pi)$, given by Legendre transformations.


\subsection{The obstruction  tower }\label{towersn}

We will now show that the tower $\Comp(A,n) \to \ldots \to \Comp(A,n)/F^{i+2} \to \ldots \to \Comp(A,n)/F^2 $ does not contain nearly as much information as first appears.

\subsubsection{Small extensions and obstructions}

\begin{definition}
 Say that a surjection $L \to M$ of DGLAs is small if the kernel $I$ satisfies $[I,L]=0$. 
\end{definition}
%

\begin{lemma}\label{obslemma}
Given a small extension $e \co L \to M$ of DGLAs 
with kernel $I$, there is a sequence
\[
\pi_0\mmc(L) \xra{e}  \pi_0\mmc(M) \xra{o_e} \H^2(I)
\]
of  sets, exact in the sense that $o_e^{-1}(0)$ is the image of $e$.
\end{lemma}
\begin{proof}
 This is well-known. 
The obstruction  map $o_e $ is given by 
\[
 o_e(\omega):=  d_L\tilde{\omega} + \half [\tilde{\omega},\tilde{\omega}]
\]
for any lift $\tilde{\omega} \in L^1$ of $\omega \in \mc(M)$.
\end{proof}
%
%
\begin{proposition}\label{obsDGLA}
For any small extension $e \co L \to M$   of DGLAs with kernel $I$, there is an obstruction  map $\ob_e\co \mmc(M) \to \mmc(I^{[1]})$ in the homotopy category of simplicial sets, with homotopy fibre $\mmc(L)$.
\end{proposition}
\begin{proof}
This is essentially the same as \cite[Theorem \ref{ddt1-robs}]{ddt1}.
Set $M'$ to be the mapping cone of $I \to L$, with  DGLA structure given by setting $[M',I^{[1]}]=0$. Then we have a surjection  $M' \to M \oplus I^{[1]}$ of  DGLAs with kernel $I$. Since the map $M' \to M$ is a small extension with acyclic kernel, it follows that $\mmc(M') \to \mmc(M)$ is a trivial fibration, hence a weak equivalence.

The obstruction map is then given by $\mmc(M') \to \mmc(I^{[1]})$, the homotopy fibre being $\mmc(L)$, as required.
\end{proof}

\begin{proposition}\label{obsles}
 Let $A$ be an abelian group object over a simplicial set $Z$, with homotopy fibre $A_z$ over $z \in Z$. If $s \co Z \to A$ is a section with homotopy vanishing locus $Y$, there is a long exact sequence
\[
 \xymatrix@R=0ex{
\cdots \ar[r] & \pi_i(Y,z)  \ar[r]& \pi_i(Z,z) \ar[r]^-s&  \pi_i(A_z,0)  \ar[r] &\pi_{i-1}(Y,z) \ar[r] &\cdots\\ 
\cdots \ar[r]& \pi_0Y \ar[r]& \pi_0Z \ar[r]^-s& \pi_0(A_?),
}
\]
of groups and sets, where the final map sends $z$ to $s(z) \in \pi_0(A_s)$
\end{proposition}
\begin{proof}
Since $Y$ is the homotopy fibre product
\[
 Z \by_{(s,0), (A\by^h_ZA)} A,
\]
the long exact sequence of homotopy 
gives 
\[
  \ldots \to \pi_i(Y,z) \to \pi_i(Z,z)  \by \pi_i(A,z) \to \pi_i(A\by^h_ZA,z)\to \ldots.
\]
Since $(A\by^h_ZA)\by^h_A\{z\}= A_z$, this simplifies to
\[
  \ldots \to \pi_i(Y,z) \to \pi_i(Z,z)  \to \pi_i(A_z,z).
\]
\end{proof}

Combining Propositions \ref{obsDGLA} and \ref{obsles} gives:
\begin{corollary}\label{obsDGLAcor}
 Given a small extension $e \co L \to M$ of DGLAs with kernel $I$, there is a canonical long exact sequence
\[
 \xymatrix@R=0ex{
\cdots  \ar[r]^-{e_*}&\pi_i\mmc(M) \ar[r]^-{o_e}& \H^{2-i}(I) \ar[r] &\pi_{i-1}\mmc(L)\ar[r]^-{e_*}&\cdots\\ 
\cdots \ar[r]^-{e_*}&\pi_1\mmc(M) \ar[r]^-{o_e}& \H^1(I)  \ar[r] &\pi_0\mmc(L) \ar[r]^-{e_*}& \pi_0\mmc(M) \ar[r]^-{o_e}& \H^{2}(I).
}
\]
\end{corollary}

\subsubsection{Towers of obstructions}

The following is  the long exact sequence of cohomology:
\begin{proposition}\label{DRobs}
For each $p$, there is a canonical long exact sequence
\begin{align*}
\ldots\to \H_{p+i-n-2}(\Omega^{p}_{A/R}) \to \pi_i(\PreSp/F^{p+1}) \to  \pi_i(\PreSp/F^{p}) \to \H_{p+i-n-3}(\Omega^{p}_{A/R}) 
\to\ldots 
\end{align*}
 of homotopy groups, where $\PreSp=\PreSp(A,n)$.
\end{proposition}


\begin{definition}\label{Mdef}
Given a compatible pair  $(\omega, \pi) \in \Comp(A,n)/F^3$ and $p \ge 0$, define the cochain complex 
$
 M(\omega,\pi,p) 
$
to be the cocone of the map 
\begin{align*}
 &(\Omega^{p}_{A/R})^{[n-p+1]} \oplus \HHom_A(\CoS_A^{p}((\Omega^1_{A/R})_{[-n-1]}),A)^{[n+1]}\\
 &\to \HHom_A(\CoS_A^{p}((\Omega^1_{A/R})_{[-n-1]}),A)^{[n+1]}
\end{align*}
given by combining                                
\[
 \Symm^{p}(\pi^{\sharp}) \co (\Omega^{p}_{A/R})^{[-p]}\to \HHom_A(\CoS_A^{p}((\Omega^1_{A/R})_{[-n-1]}),A)
\]
with the map 
\[
 \nu(\omega, \pi) - (p-1) \co \HHom_A(\CoS_A^{p}((\Omega^1_{A/R})_{[-n-1]}),A) \to \HHom_A(\CoS_A^{p}((\Omega^1_{A/R})_{[-n-1]}),A),
\]
where                              
\[
 \nu(\omega, \pi)(b):= \nu(\omega, \pi, b).
\]
\end{definition}

Because $\omega$ and $\pi$ lie in $\gr_F^2$, the description of Definition \ref{mudef} simplifies to give 
\[
 \nu(a df_1\wedge df_2, \pi)(b)= \pm a[\pi,f_1] [b,f_2] \pm a[b,f_1] [\pi,f_2]. 
\]

\begin{lemma}\label{nondegtangent}
 If $\pi$ is non-degenerate, then the projection
\[
 M(\omega,\pi,p) \to (\Omega^{p}_{A/R})^{[n-p+1]}
\]
is a quasi-isomorphism.
\end{lemma}
\begin{proof}
 On $\HHom_A((\Omega^1_{A/R})_{[-n-1]},A)$ (the case $p=1$), observe that the map  $\nu(\omega, \pi)$ is just given by $\pi^{\sharp} \circ \omega^{\sharp}$. Moreover, contraction with a $1$-form defines a derivation with respect to the commutative multiplication on polyvectors, so  $\nu(\omega, \pi)$ is the derivation  
\[
\nu(\omega, \pi)\co \widehat{\Pol}(A,n) \to \widehat{\Pol}(A,n)
\]
given on generators by $\pi^{\sharp} \circ \omega^{\sharp}$.

If $\pi$ is non-degenerate, then (by compatibility) $\pi^{\sharp} \circ \omega^{\sharp}$ is homotopic to the identity map, and thus $ \gr_F^p\nu(\omega, \pi)$ is homotopic to $p$. In particular, $\gr_F^p\nu(\omega, \pi) - (p-1) $ is a quasi-isomorphism in this case, so the projection
$
 M(\omega,\pi,p) \to (\Omega^{p}_{A/R})^{[n-p+1]}
$
is a quasi-isomorphism.
\end{proof}

If we apply 
Corollary \ref{obsDGLAcor}
to the small extensions $\gr^p_FL \to  L/F^{p+1}\to L/F^p$ for the DGLA of polyvectors,  and taking the long exact sequence of homotopy groups, we get:

\begin{proposition}\label{compatobs}
For each $p \ge 3$, there is a canonical long exact sequence
\[
 \xymatrix@C=1ex{ 
\vdots  \ar[d]^-{e_*} && \vdots  \ar[d]^-{e_*} &&\\ \pi_i(\Comp(A,n)/F^{p})  \ar[d]^-{o_e}&&\pi_1 (\Comp(A,n)/F^{p}) \ar[d]^-{o_e}&& \pi_0(\Comp(A,n)/F^{p}) \ar[d]^-{o_e} \\ \H^{2-i}M(\omega_2,\pi_2,p) \ar[d]&& \H^1 M(\omega_2,\pi_2,p) \ar[d] &&  \H^2 M(\omega_2,\pi_2,p) \\ \pi_{i-1}(\Comp(A,n)/F^{p+1}) \ar[d]^-{e_*}&&\pi_0(\Comp(A,n)/F^{p+1})  
\ar `d[dr]   `r[r]^{e_*}   `[uuur] `r[uuurr] [uurr] &&
\\ \vdots \ar@{-->} `r[ur]   `u[uuuurr]   [uuuurr] &&&&
}
\]
 of homotopy groups and sets, where $\pi_i$ indicates the homotopy group at basepoint $(\omega, \pi)$, and the target of the final map is understood to mean
\[                                                                                                                             
   {o_e}(\omega, \pi) \in \H^2 M(\omega_2,\pi_2,p).                                                                                                                                             \]
\end{proposition}
\begin{proof}
 Proposition \ref{obsDGLA} gives fibration sequences 
\begin{align*}
 \cP(A,n)/F^{p+1} \to &\cP(A,n)/F^{p} \to  \mmc (\gr_F^{p}\widehat{\Pol}(A,n)^{[n+2]} )\\
T\cP(A,n)/F^{p+1} \to &T\cP(A,n)/F^{p} \to  \mmc (\gr_F^{p}\widehat{\Pol}(A,n)^{[n+2]}[\eps] )\\
\PreSp(A,n)/F^{p+1} \to &\PreSp(A,n)/F^{p} \to  \mmc ((\Omega^{p}_{A/R})^{[n+2-p]}),
\end{align*}
for $\eps^2=0$.

We can regard these as homotopy vanishing loci for sections of trivial bundles.
Combined with 
the description of $\Comp(A,n)$ as a homotopy vanishing locus, this  
gives $\Comp(A,n)/F^{p+1}$ as a  homotopy vanishing locus  on $ \Comp(A,n)/F^{p}$. 

In more detail, we pull the sequences above back along 
\[
 (i_1, i_2) \co \Comp(A,n)/F^{p} \to  \PreSp(A,n)/F^{p} \by \cP(A,n)/F^{p},
\]
  giving  a morphism
\begin{eqnarray*}
 \Comp(A,n)/F^{p}\by \mc (\gr_F^{p}\widehat{\Pol}(A,n)^{[n+2]} \oplus (\Omega^{p}_{A/R})^{[n+2-p]})\\
 \xra{\gr^p_F(\mu-\sigma)}  \Comp(A,n)/F^{p}\by \mc(\gr_F^{p}\widehat{\Pol}(A,n)^{[n+2]} )
\end{eqnarray*}
of trivial bundles on $ \Comp(A,n)/F^{p}$. If we write $N$ for the homotopy kernel of this map, we obtain a bundle over $\Comp(A,n)/F^{p}$ equipped with a section $s$ whose homotopy vanishing locus is $\Comp(A,n)/F^{p+1} $.

It therefore remains to describe the tangents of the maps $\mu, \sigma$. Since $\sigma$ is linear on $\gr_F^{p}$, it is its own tangent, equal to $(p-1)$. To calculate the tangent of $\mu$, take $e$ with $e^2=0$, $a \in F^{p}\DR(A/R)$ and $b \in F^{p}\widehat{\Pol}(A,n)$. Then
\begin{align*}
 \mu(\omega+ ae, \pi+be) - \mu(\omega, \pi) &=   \mu(ae, \pi+be)+\mu(\omega, \pi+be) - \mu(\omega, \pi)\\
&=\mu(a,\pi)e + \nu(\omega,\pi,b)e, 
\end{align*}
because $e^2=0$.

Since terms in $F^{p+1}$ vanish, and $\mu$ preserves the filtration $F^{*+2}$, the only contributions left are
\[
\mu(a,\pi_2)e + \nu(\omega_2,\pi_2,b)e.
\]
 
At a point $(\omega, \pi)\in \Comp(A,n)/F^{p}$, the homotopy fibre $N_{\omega, \pi}$ of the bundle $N$  above is therefore just  $\mmc(M(\omega_2,\pi_2,p)^{[1]})$. 
 Then   $\Comp(A,n)/F^{p+1}$ is the homotopy vanishing locus of the obstruction map (which takes values in $N$), giving the long exact sequence of homotopy groups from Proposition \ref{obsles}.
\end{proof}

\subsection{The equivalence}

\begin{corollary}\label{compatcor1}
The canonical map
\begin{eqnarray*}
  \Comp(A,n)^{\nondeg} \to (\Comp(A,n)/F^3)^{\nondeg} \by^h_{(\PreSp(A,n)/F^3)} \PreSp(A,n)  
\end{eqnarray*}
 is a weak equivalence.
\end{corollary}
\begin{proof}
 Lemma \ref{nondegtangent} shows that for non-degenerate $\omega_2$, the map
\begin{eqnarray*}
 M(\omega_2,\pi_2,p)\to(\Omega^{p}_{A/R})^{[n-p+1]}
\end{eqnarray*}
is a  quasi-isomorphism. Propositions  \ref{DRobs} and \ref{compatobs} thus combine to show that the maps
\begin{eqnarray*}
 \Comp(A,n)^{\nondeg}/F^{p+1} \to (\Comp(A,n)^{\nondeg}/F^{p})\by^h_{(\PreSp(A,n)/F^{p})}(\PreSp(A,n)/F^{p+1})
\end{eqnarray*}
are weak equivalences for all $p \ge 3$. We then just take the limit over all $p$.
\end{proof}

\begin{proposition}\label{level0prop}
The canonical map
 \begin{eqnarray*}
  \Comp(A,n)^{\nondeg}/F^3 &\to& \Sp(A,n)/F^3    
\end{eqnarray*}
 is a weak equivalence.
\end{proposition}
\begin{proof}
We may apply Proposition \ref{obsles} to the  definition of $\Comp$ as a homotopy vanishing locus, giving a long exact sequence  
\[
 \xymatrix@C=1ex
{
\vdots \ar[d] && 
 \\  \pi_i(\Comp(A,n)/F^3 ,(\omega, \pi))  \ar[d] && \pi_0(\Comp(A,n)/F^3) \ar[d]\\ \pi_i(\PreSp(A,n)/F^3, \omega) \by \pi_i( \cP(A,n)/F^3, \pi)  \ar[d]^{\mu-\sigma} &&  \pi_0(\PreSp(A,n)/F^3) \by \pi_0(\cP(A,n)/F^3) \ar[d]^{\mu-\sigma}\\  \pi_i(T\cP(A,n)/F^3 ,0)  
\ar@{-->} `d[dr]   `r[r]   `[uuur] `r[uuurr] [uurr] 
&& \pi_0(T_{\omega, \pi}\cP(A,n)/F^3) \\ 
&&
}
\]
%
of groups and sets.

We just have 
\begin{eqnarray*}
 \pi_i(\PreSp(A,n)/F^3, \omega)&=& \H_{i-n}(\Omega^{2}_{A/R})\\
\pi_i(\cP(A,n)/F^3)&=&  \H^{n+2-i}(\gr_F^2\widehat{\Pol}(A,n))\\
\pi_i(T_{\omega,\pi} \cP(A,n)/F^3)&=&  \H^{n+2-i}(\gr_F^2\widehat{\Pol}(A,n)),
\end{eqnarray*}
with the map $\mu -\sigma$ given as in  Definition \ref{Mdef} by combining  $   \Symm^{2}(\pi^{\sharp}) \co (\Omega^{2}_{A/R})_{[2]}\to \HHom_A(\CoS_A^{2}((\Omega^1_{A/R})_{[-n-1]}),A)$ with 
 the map $\nu(\omega, \pi) - 1$. 

As in Lemma \ref{nondegtangent}, when $\pi$ is non-degenerate, the map  $\gr_F^2\nu(\omega, \pi) - 1 $ is a quasi-isomorphism, inducing isomorphisms
\begin{eqnarray*}
 \pi_i(\Comp(A,n)/F^3 ,(\omega, \pi))  \cong \pi_i(\PreSp(A,n)/F^3, \omega)
\end{eqnarray*}
for all $i>0$. For $i=0$, the argument of Example \ref{compatex2} works equally well modulo $F^3$; combined with the exact sequence above, it shows that the locus
\begin{eqnarray*}
\pi_0(\Comp(A,n)/F^3)^{\nondeg} \into \pi_0(\PreSp(A,n)/F^3)= \H^n(\Omega^2_{A/R})
\end{eqnarray*}
corresponds to the non-degenerate elements  $ \pi_0(\Sp(A,n)/F^3)$.
 \end{proof}

\begin{corollary}\label{compatcor2}
The canonical maps
\begin{eqnarray*}
  \Comp(A,n)^{\nondeg} &\to& \Sp(A,n)  \\   
    \Comp(A,n)^{\nondeg} &\to&  \cP(A,n)^{\nondeg}           
\end{eqnarray*}
 are weak equivalences.
\end{corollary}
\begin{proof}
 For the first map, just combine Corollary \ref{compatcor1} with Proposition \ref{level0prop}. The second map is given by Proposition \ref{compatP1}.
\end{proof}

\section{Compatible structures on derived DM stacks}\label{DMsn}

We will now study symplectic and Poisson structures on derived  $N$-stacks. Rather than studying a single CDGA as in 
  Corollary \ref{compatcor2}, will will establish a similar result for  strings of \'etale morphisms between non-degenerate structures. 
Following the philosophy of \cite{stacks2}, we can represent a derived Artin or DM $N$-stack $\fX$ as a simplicial derived affine scheme $X_{\bt}$ satisfying various conditions. Thus $X_{\bt}$ is $\Spec O(X)^{\bt}$ for a cosimplicial diagram $O(X)^{\bt}$ of CDGAs. For a derived DM stack, the morphisms in this diagram are \'etale, and descent results  will show that Poisson structures are independent of the choice of resolution $X_{\bt}$.  A modified construction will be given  for derived Artin stacks in \S \ref{Artinsn}.  

\begin{definition}
Write $dg\CAlg(\Q)$ for the category of chain complexes over $\Q$ with graded-commutative multiplication. We will refer to its objects as chain CDGAs (in contrast with the more common, but equivalent, cochain CDGAs). Given $R \in dg\CAlg(\Q) $, we write $dg\CAlg(R) $ for the category of  graded-commutative $R$-algebras in chain complexes. 
\end{definition}

Throughout this section, we   fix $R \in dg\CAlg(\Q)$. 

\subsection{Compatible structures on diagrams}\label{DMdiagramsn} 

\subsubsection{Definitions}

Given a small category $I$, an $I$-diagram $A$ of chain CDGAs over $R$, and   $A$-modules $M,N$ in $I$-diagrams of chain complexes, we can define the  cochain complex $\HHom_A(M,N)$ to be the equaliser of the obvious diagram
\[
\prod_{i\in I} \HHom_{A(i)}(M(i),N(i)) \implies \prod_{f\co i \to j \text{ in } I}   \HHom_{A(i)}(M(i),f_*N(j)).
\]
All the constructions of \S \ref{poisssn} then adapt immediately; in particular, we can define 
\[
 \widehat{\Pol}(A,n):= \HHom_A(\Co\Symm_A((\Omega^1_{A/R})_{[-n-1]}),A),
\]
leading to a space $\cP(A,n)$ of Poisson structures.

In order to ensure that this has the correct homological properties, we now consider categories  of  the form $[m]= (0 \to 1 \to \ldots \to m)$.
\begin{lemma}\label{calcTOmegalemma}
 If $A$ is an $[m]$-diagram in   chain CDGAs over $R$  which is cofibrant and  fibrant for the injective model structure (i.e. each $A(i)$ is cofibrant and the maps $A(i) \to A(i+1)$ are surjective), then $\HHom_A(\CoS_A^k\Omega^1_A,A)$ is a model for the derived $\Hom$-complex $\oR\HHom_A(\oL\CoS_A^k\oL\Omega^1_A,A)$, and $ \HHom_A(A,\Omega^p_A) \simeq \ho\Lim_i \oL\Omega^p_{A(i)}$.
\end{lemma}
\begin{proof}
Because $A(i)$ is cofibrant, $\Omega^p_{A(i)}$ is cofibrant as an $A(i)$-module, so   the complex  $\oR\HHom_A(\oL\CoS_A^k\oL\Omega^1_A,A)$ is the homotopy limit of the diagram
\[
 \HHom_{A(0)}(C(0),A(0))\to  \HHom_{A(0)}(C(0),A(1)) \la \HHom_{A(1)}(C(1),A(1)) \to \ldots,
\]
where we write $C(i) := \CoS_{A(i)}^k\Omega^1_{A(i)}$.
Since the maps $A(i) \to A(i+1)$ are surjective, the maps $\HHom_{A(i)}(\CoS_{A(i)}^k\Omega^1_{A(i)},A(i))\to  \HHom_{A(i)}(\Omega^p_{A(i)},A(j))$ are so, and thus the homotopy limit is calculated by the ordinary limit, which is precisely $\HHom_A(\CoS_{A}^k\Omega^1_A,A) $.

Because $[m]$ has initial object $0$,  $ \HHom_A(A,\Omega^p_A) \cong \HHom_{A(0)}(A(0), \Omega^p_{A(0)})= \Omega^p_{A(0)}$. For the same reason, $ \ho\Lim_i \Omega^p_{A(i)}\simeq \Omega^p_{A(0)}$.
\end{proof}

For such diagrams,  we can also adapt all the constructions of \S \ref{prespsn} immediately by taking inverse limits, so setting
\[
 \PreSp(A,n):=  \PreSp(A(0),n)= \Lim_{i\in [m]}  \PreSp(A(i),n)
\]
for any $[m]$-diagram $A$ of chain CDGAs over $R$. 

We then adapt the constructions of \S \ref{compsn}, defining
\[
 \mu \co \PreSp(A,n) \by \cP(A,n) \to T\cP(A,n)
\]
by setting $\mu(\omega, \Delta)(i):= \mu(\omega(i), \Delta(i)) \in T\cP(A(i),n)$ for $i \in [m]$, and letting $ \Comp(A,n)$ be the homotopy vanishing locus of
\[
(\mu - \sigma) \co  \PreSp(A,n) \by \cP(A,n) \to  \pr_2^*T\cP(A,n).
\]

The obstruction functors and their towers from \S \ref{towersn} also adapt immediately to $[m]$-diagrams, giving the obvious analogues of the obstruction spaces defined in terms of 
\[
\HHom_A(\CoS^p_A((\Omega^1_{A/R})_{[-n-1]}),A) \quad\text{ and }\quad \Omega^p_{A(0)}.
\]

\subsubsection{Functors and descent}\label{descentsn}

\begin{lemma}\label{calcTlemma2}
If $D=(A\to B)$ is a fibrant cofibrant  $[1]$-diagram in $dg\CAlg(R)$ which is formally \'etale in the sense that the map
\[
 \Omega_{A}^1\ten_{A}B \to \Omega_{B}^1
\]
is a quasi-isomorphism, then the map    
\[
\HHom_D(\CoS_D^k\Omega^1_D,D) \to \HHom_{A}(\CoS_{A}^k\Omega^1_{A},A),
\]
is a quasi-isomorphism.
\end{lemma}
\begin{proof}
 This follows immediately from the proof of Lemma \ref{calcTOmegalemma}, using the quasi-isomorphism
\[
\HHom_{B}(\CoS_{B}^k\Omega^1_{B},B)\to \HHom_{A}(\CoS_{A}^k\Omega^1_{A},B).
\]
\end{proof}
For a similar result, see \cite[Lemma 1.4.13]{CPTVV}.

\begin{definition}
 Write  $dg\CAlg(R)_{c, \onto}\subset dg\CAlg(R) $ for the subcategory with all cofibrant chain CDGAs  over $R$ as objects, and only surjective morphisms.
\end{definition}

We already have functors $\PreSp(-,n)$ and $\Sp(-,n)$ from  $dg\CAlg(R)$ to $s\Set$, the category of  simplicial sets, mapping quasi-isomorphisms in $dg\CAlg(R)_{c}$ to weak equivalences. Poisson structures are only functorial with respect to formally \'etale morphisms, in an $\infty$-functorial sense which we now make precise. 

Observe that, writing $F$ for any  of the constructions  $\cP(-,n)$, $\Comp(-,n)$, $\PreSp(-,n)$,  and the associated filtered and graded functors, applied to $[m]$-diagrams in $dg\CAlg(R)_{c, \onto}$,  we have:

\begin{properties}\phantomsection\label{Fproperties}

 \begin{enumerate}
 \item the maps from $F(A(0)\to \ldots \to A(m))$ to 
\[
 F(A(0)\to A(1))\by_{F(A(1))}^h F(A(1)\to A(2))\by^h_{F(A(2))}\ldots \by_{F(A(m-1))}^hF(A(m-1) \to A(m))
\]
 are weak equivalences;

\item if the $[1]$-diagram $A \to B$ is a quasi-isomorphism, then the natural maps from $F(A \to B)$ to $F(A)$ and to $F(B)$ are weak equivalences. 

\item if the $[1]$-diagram $A \to B$ is formally \'etale, then the natural map from $F(A \to B)$ to $F(A)$ is a  weak equivalence. 
\end{enumerate}
These properties follow from Lemmas \ref{calcTOmegalemma} and \ref{calcTlemma2}, 
 together with  the obstruction calculus of \S \ref{towersn} extended to diagrams.
\end{properties}

Property \ref{Fproperties}.1 ensures that the simplicial classes $\coprod_{ A \in B_m dg\CAlg(R)_{c, \onto}} F(A)$ fit together to give a complete Segal space $\int F$ over the nerve $Bdg\CAlg(R)_{c, \onto} $. Taking complete Segal spaces \cite[\S 6]{rezk}  as our preferred model of $\infty$-categories:

\begin{definition}\label{LintFdef}
Define $\oL dg\CAlg(R)_{c, \onto}$, $\oL dg\CAlg(R)$, $\oL\int F$, and $\oL s\Set$   to be the $\infty$-categories obtained by localising the respective 
categories at quasi-isomorphisms or weak equivalences.
\end{definition}

Property \ref{Fproperties}.2  ensures that the homotopy fibre of $\oL\int F \to  \oL dg\CAlg(R)_{c, \onto}$ over $A$ is still just the simplicial set $F(A)$, regarded as an $\infty$-groupoid.

Since the surjections in $dg\CAlg(R)$ are the fibrations, the inclusion   $\oL dg\CAlg(R)_{c, \onto} \to \oL dg\CAlg(R)$ is a weak equivalence, and we may regard $\oL\int F $ as an $\infty$-category over $\oL dgCAlg(R) $.

Furthermore,  Property \ref{Fproperties}.3 implies that the $\infty$-functor $\oL\int F \to \oL dg\CAlg(R) $ is a co-Cartesian fibration when we restrict to the $2$-sub-$\infty$-category $\oL dg\CAlg(R)^{\et} $ of homotopy formally \'etale morphisms, giving:

\begin{definition}\label{inftyFdef}
When $F$ is any of   the constructions above, we define 
\[
 \oR F \co \oL dg\CAlg(R)^{\et} \to  \oL s\Set
\]
to be the $\infty$-functor whose Grothendieck construction is $\oL\int F  $.
\end{definition}

In particular, the observations above ensure that 
\[
 (\oR F)(A) \simeq F(A)
\]
for all cofibrant chain CDGAs $A$ over $R$.

An immediate consequence of Corollary \ref{compatcor2} is that the canonical maps
\begin{eqnarray*}
  \oR\Comp(-,n)^{\nondeg} &\to& \oR\Sp(-,n)  \\   
    \oR\Comp(-,n)^{\nondeg} &\to&  \oR\cP(-,n)^{\nondeg}           
\end{eqnarray*}
 are natural weak equivalences of $\infty$-functors on $\oL dg\CAlg(R)^{\et}$. 

\subsection{Derived  \tps{$N$}{N}-hypergroupoids}\label{hgpdsn}

\subsubsection{Background}

We now require our chain CDGA $R$ over $\Q$ to be concentrated in non-negative chain degrees, and write   $dg_+\CAlg(R)\subset dg\CAlg(R) $ for the full subcategory of chain CDGAs  which are concentrated in non-negative chain degrees. We denote the opposite category to $dg_+\CAlg(R) $ by $DG^+\Aff_R$. Write $sDG^+\Aff_R$ for the category of simplicial diagrams in $DG^+\Aff_R $. A morphism in $DG^+\Aff_R $ is said to be a fibration if it is given by a cofibration in the opposite category $dg_+\CAlg(R)$ (which in turn just means that it is a retract of a quasi-free map). 

As in \cite{hag2}, a morphism $f:A \to B$ in $dg_+\CAlg(R)$ is said to be smooth (resp. \'etale) if $\H_0(f): \H_0A \to \H_0B$ is smooth, and  the maps $\H_n(A)\ten_{\H_0(A)}\H_0(B) \to \H_n(B)$ are isomorphisms for all $n$. The associated map $\Spec B \to \Spec A$ in $ DG^+\Aff_R$ is said to be surjective if $\Spec \H_0B \to \Spec \H_0A$ is so.

\begin{definition}
For $m \ge 0$, the combinatorial $m$-simplex $\Delta^m \in s\Set$ is characterised by the property that $\Hom_{s\Set}(\Delta^m, K) \cong K_m$ for all simplicial sets $K$. Its boundary $\pd\Delta^m \subset \Delta^m$ is given by $\bigcup_{i=0}^m\pd^i(\Delta^{m-1})$ (taken to include the case $\pd\Delta^0=\emptyset$), and for $m \ge 1$ the $k$th horn $\Lambda^{m,k}$ is given by $\bigcup_{\substack{0 \le i \le m\\ i \ne k}}\pd^i(\Delta^{m-1})$.
\end{definition}

\begin{definition}\label{mn}
Given a simplicial set $K$ and a simplicial object  $X_{\bt}$ in a complete category $\C$, we follow \cite[Proposition VII.1.21]{sht} in defining the $K$-matching object in $\C$ by 
$$
M_KX:= \Hom_{s\Set}(K, X).
$$
Note that for finite simplicial sets $K$, the  matching object $M_KX$ still exists even if $\C$ only contains finite limits.
\end{definition}

Explicitly, the matching object $M_{\pd \Delta^m}(X)$ is given by the equaliser of a diagram
\[
  \prod_{0\le i \le m} X_{m-1} \implies    \prod_{0\le i<j \le m} X_{m-2},    
\]
while the partial matching object $  M_{\Lambda^m_k} (X)$
 is given by the equaliser of a diagram
\[
  \prod_{\substack{0\le i \le m\\i \ne k}} X_{m-1} \implies    \prod_{\substack{0\le i<j \le m\\i,j \ne k}} X_{m-2}.     
\]

We now recall some results from \cite{stacks2}. 

\begin{definition} 
Given $Y \in sDG^+\Aff_R$, a DG Artin (resp. DM) $N$-hypergroupoid $X$ over $Y$ is a morphism $X \to Y$ in $sDG^+\Aff_R$ for which:
\begin{enumerate}
 \item the matching maps
$$
X_m \to M_{\pd \Delta^m} (X)\by_{M_{\pd \Delta^m} (Y)}Y_m 
$$
are fibrations for all $m\ge 0$;

\item  the   partial matching maps
$$
X_m \to M_{\Lambda^m_k} (X)\by_{M^h_{\Lambda^m_k} (Y)}^hY_m 
$$
are  smooth (resp. \'etale) surjections for all $m \ge 1$ and $k$, and are weak equivalences for all $m>N$ and all $k$.
\end{enumerate}

A morphism $X\to Y$ in $sDG^+\Aff_R$ is a  trivial DG Artin (resp. DM)  $N$-hypergroupoid if and only if the matching maps
$$
X_m \to M_{\pd \Delta^m} (X)\by_{M_{\pd \Delta^m} (Y)}Y_m 
$$
of Definition \ref{mn}
are surjective smooth (resp. \'etale) fibrations  for all  $m$, and are weak equivalences for all $m\ge n$.
\end{definition}

The following is \cite[Theorem \ref{stacks2-bigthm} and Corollary \ref{stacks2-Dequivcor}]{stacks2}, as spelt out in \cite[Theorem \ref{stacksintro-dbigthm}]{stacksintro}:

\begin{theorem}\label{dbigthm}
%
The $\infty$-category of strongly quasi-compact $N$-geometric derived Artin (resp. DM) stacks  over $R$ is given by localising the category  of DG Artin (resp.) $N$-hypergroupoids  over $R$  at the class of trivial relative  DG Artin (resp. DM)  $N$-hypergroupoids.
\end{theorem}

Given a  DG Artin (resp. DM)  $N$-hypergroupoid  $X$, we denote the associated $N$-geometric derived Artin (resp. DM) stack by $X^{\sharp}$.

\begin{definition}
Given  $X \in sDG^+\Aff_R$, define $cdg\Mod(X)$ to be the  category of $O(X)$-modules in cosimplicial diagrams of chain complexes. 
We say that a morphism $M \to N$ in $cdg\Mod(X)$ is a  weak equivalence 
if it
induces quasi-isomorphisms 
$M^i \to N^i$ of chain complexes 
in each cosimplicial level.
\end{definition}

The following is \cite[Proposition \ref{stacks2-qcohequiv} and Remarks \ref{stacks2-hcartrks}]{stacks2}:
\begin{proposition}\label{qcohequiv}
For a DG Artin $N$-hypergroupoid $X$,  the $\infty$-category of  quasi-coherent complexes (in the sense of \cite[\S 5.2]{lurie})
on the $n$-geometric derived stack $X^{\sharp}$ is equivalent to the localisation at weak equivalences of the full subcategory $dg\Mod(X)_{\cart}$ of $cdg\Mod(X)$ consisting of modules $M$ which are homotopy-Cartesian in the sense that   the maps
\[
 \pd^i\co \oL\pd_i^*M^{m-1} \to M^m
\]
are quasi-isomorphisms for all $i$ and $m$.
\end{proposition}

\begin{definition}\label{ulinedef}
 We make $cdg\Mod(X)$  into a simplicial category by setting (for $K \in s\Set$)
$$
(M^K)^n: = (M^n)^{K_n}, 
$$
as an $O(X)^n$-module in chain complexes. This has a left adjoint $M \mapsto M\ten K$. 

Given $M \in cdg\Mod(X)$, define $\uline{M} \in (cdg\Mod(X))^{\Delta}$ to be the cosimplicial diagram given in cosimplicial level $n$ by $M\ten \Delta^n$.
\end{definition}

Combining \cite[Definition \ref{stacks2-cotdef} and Corollary \ref{stacks2-loopcot}]{stacks2} gives:
\begin{definition}\label{cotdef}
For a DG Artin $N$-hypergroupoid $X$,  define the cotangent complex $\bL^{X/S} \in cdg\Mod(X)$ by  $\bL^{X/S}:=  \Tot N_c^{\le N}\underline{\Omega(X/S)}$, where $N_c$ denotes cosimplicial conormalisation, and $\Tot$ is the direct sum total functor from cochain chain complexes to chain complexes.
\end{definition}

By \cite[Lemma \ref{stacks2-Lcart} and Corollary \ref{stacks2-cotgood}]{stacks2}, this is homotopy-Cartesian and   recovers the usual cotangent complexes in derived algebraic geometry under the correspondence of Proposition \ref{qcohequiv}.

\subsubsection{Poisson and symplectic structures}

Take a DG Deligne--Mumford  $N$-hypergroupoid $X$ over $R$. Because the face maps $\pd^i \co  \Delta^{m-1} \to \Delta^m$  are weak equivalences, the morphisms $\pd_i\co  X_m \to X_{m-1}$ are all \'etale. Since $\sigma_i \co X_{m-1} \to X_m$ has left inverse $\pd_i$, it follows that $\sigma_i$ is also \'etale. 

Therefore $O(X)$ is a functor from $\Delta^{\op}$ to  $dg\CAlg(R)^{\et}$. 

\begin{definition}\label{inftyFXdef}
 For any of the functors $F$ in  Definition \ref{inftyFdef}, write 
\[
 F(X):= \ho\Lim_{j \in \Delta} \oR F(O(X_j)).
\]
\end{definition}
This gives us spaces
$\cP(X,n)$, $\PreSp(X,n)$ and $\Comp(X,n)$
 of Poisson structures, pre-symplectic structures and compatible pairs.
Explicitly, strictification theorems then  imply  that an element of $  \cP(X,n)$ can be represented by a cosimplicial diagram $P$ of $\hat{P}_{n+1}$-algebras equipped with a weak equivalence from $O(X)$ to the cosimplicial chain CDGA underlying $P$. 

\begin{proposition}\label{inftyFXwell}
If $Y \to X$ is a trivial DG DM hypergroupoid, then the morphism 
\[
 F(X) \to F(Y)
\]
is an equivalence for any of the constructions $F= \cP, \Comp, \PreSp$.
\end{proposition}
\begin{proof}
The morphism $F(X) \to F(Y)$ exists because the maps $Y_j \to X_j$ are all \'etale. By Propositions \ref{DRobs}, \ref{compatobs} and the corresponding statement for $\cP$, it suffices to prove that the morphisms
\begin{align*}
\Omega^p_{X_j/R} &\to \Omega^p_{Y_j/R},\\ 
 \HHom_{O(X_j)}(\CoS^p((\Omega^1_{X_j/R})_{[-n-1]}), O(X_j))  &\to  \HHom_{O(Y_j)}(\CoS^p((\Omega^1_{Y_j/R})_{[-n-1]}), O(Y_j)) 
\end{align*}
induce quasi-isomorphisms on the homotopy limits over $j \in \Delta$. 
This follows by faithfully flat descent for quasi-coherent sheaves, since $\fX :=X^{\sharp} \simeq Y^{\sharp}$, the expressions reducing to $\oR \Gamma(\fX, \oL\Omega^p_{\fX/R})$ and $\oR\HHom_{\O_{\fX}}(\CoS^p((\bL_{\fX/R})_{[-n-1]}), \O_{\fX}) $, respectively. 
 \end{proof}

Thus the following is well-defined:
\begin{definition}\label{DMFdef}
 Given a strongly quasi-compact DG DM $N$-stack $\fX$,  define the spaces  $\cP(\fX,n)$, $\Comp(\fX,n)$, $\Sp(\fX,n)$  to be  the spaces
$
 \cP(X,n), \Comp(X,n), \Sp(X,n)
$
for any DG DM $N$-hypergroupoid $X$ with $X^{\sharp} \simeq \fX$.
\end{definition}

By the results of  \S \ref{descentsn}, an immediate consequence of Corollary \ref{compatcor2} is then: 
  \begin{theorem}\label{DMthm}
 There are natural weak equivalences
\[
 \Sp(\fX,n) \la \Comp(\fX,n)^{\nondeg}\to \cP(\fX,n)^{\nondeg}.
\]
\end{theorem}

\section{Compatible structures on derived Artin stacks}\label{Artinsn}

We now show how to extend the comparisons above to derived Artin $N$-stacks over a fixed chain CDGA $R \in dg_+\CAlg(\Q)$. The main difficulty is in establishing an appropriate notion for shifted Poisson structures in this setting. It is tempting just to consider DG Artin $N$-hypergroupoids equipped with Poisson structures levelwise, with weak equivalences generated by trivial DG Artin $N$-hypergroupoids. With this approach, it is not easy to compute the obstructions arising, and it seems unlikely that they will be  related to the cotangent complex in the desired way. 

\begin{remark}
 By contrast, deformations of a DG Artin hypergroupoid $X$ do recover all deformations of the associated derived Artin stack, at least on the level of objects (cf. \cite[Theorem \ref{stacks2-deformstack}]{stacks2}). The key in that case was the comparison between  $\Omega^1(X/R)$ and the cotangent complex $\bL_{X/R}$ in \cite[Lemma \ref{stacks2-deform2}]{stacks2}, which relies on strong boundedness properties. The shifts involved in the definition of Poisson structures  preclude  any analogous result for the obstruction tower associated to a semi-strict Poisson structure.  
\end{remark}

\subsection{Stacky CDGAs}\label{stackyCDGAsn}

Our  solution is to replace   Artin $N$-hypergroupoids in  CDGAs  with DM hypergroupoids in a suitable category 
of  graded-commutative algebras in double complexes, which we refer to as stacky CDGAs.
 The idea is to globalise one of the intermediate steps \cite[Theorem \ref{ddt1-dequiv}]{ddt1} in the proof 
of the equivalence between formal moduli problems and DGLAs.

Given a stacky CDGA $A$, the definition of an $n$-shifted Poisson structure is fairly obvious: it is a Lie bracket of total cochain degree $-n$,
 or rather an $L_{\infty}$-structure in the form of a  sequence $[-]_m$ of $m$-ary operations of cochain degree $1-(n+1)(m-1)$. However, the precise formulation  (Definition \ref{bipoldef}) is quite subtle, involving  lower bounds on the cochain degrees of the operations.


 When working with stacky CDGAs, we write the double complexes as chain cochain complexes, enabling us to distinguish easily between derived (chain) and stacky (cochain) structures: 
\begin{definition}
 Define a chain cochain complex $V$ over $\Q$ to be  a bigraded $\Q$-vector space $V= \bigoplus_{i,j}V^i_j$, equipped with square-zero linear maps $\pd \co V^i_j \to V^{i+1}_j$ and $\delta \co V^i_j \to V^i_{j-1}$ such that $\pd\delta + \delta \pd =0$.

There is an obvious tensor product operation $\ten$ on this category, and a stacky CDGA is then defined to be a chain cochain complex $A$ equipped with a commutative product $A\ten A \to A$ and unit $\Q \to A$. 

We  regard all chain complexes as   chain cochain  complexes $V= V^0_{\bt}$.
Given a chain  CDGA $R$, a stacky CDGA over $R$ is then a morphism $R \to A$ of stacky CDGAs. We write $DGdg\CAlg(R)$ for the category of  stacky CDGAs over $R$, and $DG^+dg\CAlg(R)$ for the full subcategory consisting of objects $A$ concentrated in non-negative cochain degrees.
\end{definition}

\begin{definition}
 Say that a morphism $U \to V$ of chain cochain complexes is a levelwise quasi-isomorphism if $U^i \to V^i$ is a quasi-isomorphism for all $i \in \Z$. Say that a morphism of stacky CDGAs is a levelwise quasi-isomorphism if the underlying morphism of chain cochain complexes is so.
\end{definition}

\begin{lemma}\label{bicdgamodel}
There is a cofibrantly generated model structure on stacky CDGAs over $R$ in which fibrations are surjections and weak equivalences are levelwise quasi-isomorphisms. 
\end{lemma}
\begin{proof}
We first prove the corresponding statement for  chain cochain complexes over $\Q$. Let $D^i$ be the chain complex $k_{[-i]} \to k_{[1-i]}$ and $S^i:= k_{[-i]}$, so  there is an obvious map $S^i \to D^i$. Write $D_i$ for the cochain complex $k^{[1-i]} \to k^{[-i]}$.  We can then  take the  set $I$ of generating cofibrations to consist of the morphisms 
\[
 \{D_i\ten S^j \to D_i\ten D^j\}_{i,j \in \Z},
\]
with the set $J$ of generating trivial cofibrations given by 
\[
  \{0 \to D_i\ten D^j\}_{i,j \in \Z}.
\]

It is straightforward to verify that this satisfies the conditions of \cite[Theorem 2.1.19]{hovey}. The forgetful functor from stacky CDGAs to chain cochain complexes has a left adjoint $V \mapsto R\ten \Symm_{\Q}V$, which transfers the model structure by \cite[Theorem 11.3.2]{Hirschhorn}. 
\end{proof}

As described for instance in  \cite[Definition \ref{ddt1-nabla}]{ddt1}, there is a denormalisation functor $D$ from non-negatively graded cochain CDGAs to cosimplicial algebras, with 
 left adjoint $D^*$. 
Given a cosimplicial chain CDGA $A$, $D^*A$ is then a stacky CDGA, with $ (D^*A)^i_j=0$ for $i<0$. The functor $D$ satisfies $(DB)^m \cong \bigoplus_{i=0}^m (B^i)^{\binom{m}{i}}$, with multiplication coming from  the shuffle product, so in particular the iterated codegeneracy map $(DB)^m \to B^0$ is always an $m$-nilpotent extension. As a result, the left adjoint $D^*$ factors through the functor sending $A$ to its pro-nilpotent completion over $A^0$. 


\begin{lemma}\label{Dstarlemma}
The  functor $D^*$  
is a left Quillen functor from the Reedy model structure on cosimplicial chain CDGAs to the model structure of Lemma \ref{bicdgamodel}. 
\end{lemma}
\begin{proof}
The denormalisation functor $D$ on non-negatively graded cochain complexes extends to all cochain complexes by composing with brutal truncation. The right adjoint to $D^*$ is given by applying $D$ to the cochain index of a  stacky CDGA $A$. Since fibrations and weak equivalences in the Reedy model structure are levelwise surjections and levelwise quasi-isomorphisms, it follows immediately that $D$ is right Quillen, so $D^*$ is left Quillen.
\end{proof}

To any  DG Artin  $N$-hypergroupoid $X$ over $R$, we can then associate the stacky CDGA $D^*O(X)$. This behaves well because DG Artin hypergroupoids are Reedy fibrant, so there is no need to replace $D^*$ with an associated left-derived functor. Because $DA$ is always a nilpotent extension of $A^0$, we cannot recover $X$ from the levelwise quasi-isomorphism class of $D^*O(X)$, but we can recover $X_0$ and the formal completion of $X$ along $X_0$. Thus $D^*O(X)$ will play the same role as the formal affine derived stacks of \cite{CPTVV}. 
%

For any DG Artin  $N$-hypergroupoid $X$, the diagram $X^{\Delta^m}$ is another   DG Artin  $N$-hypergroupoid resolving the same stack, and we will then consider the cosimplicial stacky CDGA $m \mapsto D^*O(X^{\Delta^m})$, which we can think of as a kind of  DM $N$-hypergroupoid in stacky CDGAs.

\begin{example}\label{DstarBG}
If  $Y$ is a derived affine scheme equipped with an action of a smooth affine group scheme $G$, then the nerve $X:=B[Y/G]$ is a hypergroupoid resolving  the derived Artin $1$-stack $[Y/G]$. Since  $B[Y/G]_i = Y \by G^i$, the simplicial derived scheme $X^{\Delta^m}$ is given  by  $B[Y \by G^m/G^{m+1}]$, with action  
\[
 (y,h_1, \ldots,h_m)(g_0, \ldots,g_m)= (y g_0, g_0^{-1}h_1g_1,g_1^{-1}h_2g_2, \ldots, g_{m-1}^{-1}h_mg_m).
\]

Taking Reedy fibrant replacement of $X$ gives us a DG Artin hypergroupoid $X'$ and a levelwise quasi-isomorphism $X \to X'$. Since  the face maps of
$X$ are smooth, the map $D^*O((X')^{\Delta^{\bt}}) \to D^*O(X^{\Delta^{\bt}})$ is a levelwise quasi-isomorphism of cosimplicial stacky CDGAs. Thus Poisson structures on the stack $[Y/G]$ can be defined in terms of polyvectors on the cosimplicial stacky CDGA $D^*O(X^{\Delta^{\bt}}) $ when $Y$ is fibrant.

The completion $O(\hat{X})$ of $O(X)$ over $O(X_0)$ is a cosimplicial  CDGA with $O(\hat{X})^i= O(Y)\llbracket(\g^{\vee})^{\bigoplus i}\rrbracket$, for $\g$ the Lie algebra of $G$, where we write $A\llbracket M \rrbracket$ for the $(M)$-adic completion of $A\ten_R\Symm_R(M)$. Thus the stacky CDGA $D^*O(X)$ is the Chevalley--Eilenberg double complex 
\[
O([Y/\g]):=(  O(Y) \xra{\pd} O(Y)\ten \g^{\vee} \xra{\pd} O(Y)\ten \Lambda^2\g^{\vee}\xra{\pd} \ldots) 
\]
of $\g$ with coefficients in the chain $\g$-module $O(Y)$. 

The same calculation for $[Y \by G^m/G^{m+1}]$ 
%
shows that 
the cosimplicial stacky CDGA $D^*O(X^{\Delta^{\bt}})=D^*O((B[Y/G])^{\Delta^{\bt}})$ is given by a diagram 
\[
 \xymatrix@1{ O([Y/\g]) \ar@<1ex>[r] \ar@<-1ex>[r] & \ar@{.>}[l] O([Y \by G/\g^{\oplus 2}]) \ar[r] \ar@/^/@<0.5ex>[r] \ar@/_/@<-0.5ex>[r] & \ar@{.>}@<0.75ex>[l] \ar@{.>}@<-0.75ex>[l]   
O([Y \by G^2/\g^{\oplus 3}]) \ar@/^1pc/[rr] \ar@/_1pc/[rr] \ar@{}[rr]|{\cdot} \ar@{}@<1ex>[rr]|{\cdot} \ar@{}@<-1ex>[rr]|{\cdot} &&  {\phantom{E}}\cdots .}
\]


Beware that maps $(Y',G')\to (Y,G)$ only induce  levelwise  quasi-isomorphisms $D^*O(B[Y'/G']) \to D^*O(B[Y/G])$ when $Y'$ is weakly equivalent to $Y$ and $\g'$ isomorphic to $\g$  --- it is not enough for the Lie algebra cohomology groups to be isomorphic.
\end{example}

\subsubsection{Modules over stacky CDGAs}

For now, we fix a cofibrant stacky CDGA $A$ over a $\Q$-CDGA $R$. 
\begin{definition}
 Given a chain cochain complex $V$, define the cochain complex $\hat{\Tot} V \subset \Tot^{\Pi}V$ by
\[
(\hat{\Tot} V)^m := (\bigoplus_{i < 0} V^i_{i-m}) \oplus (\prod_{i\ge 0}   V^i_{i-m})
\]
with differential $\pd \pm \delta$.
\end{definition}

An alternative description of $\hat{\Tot} V$ is as the completion of $\Tot V$ with respect to the filtration $ \{\Tot \sigma^{\ge m }V\}_m$, where $\sigma^{\ge m}$ denotes brutal truncation in the cochain direction. In fact, we can write
\[
 \Lim_m \LLim_n \Tot( (\sigma^{\ge - n }V)/(\sigma^{\ge m }V)) = \hat{\Tot} V =  \LLim_n \Lim_m\Tot( (\sigma^{\ge -n }V)/(\sigma^{\ge m }V)).
\]
 The latter description also shows that there is a canonical map $(\hat{\Tot} U)\ten( \hat{\Tot} V) \to \hat{\Tot} (U \ten V)$ --- the same is not true of $\Tot^{\Pi}$ in general.

Write $DGdg\Mod(A)$ for the category of $A$-modules in chain cochain complexes. When $A \in DG^+dg\CAlg(\Q)$, write $DG^+dg\Mod(A) \subset DGdg\Mod(A) $ for the full subcategory of objects concentrated in non-negative cochain degrees.

\begin{definition}
 Given $A$-modules $M,N$ in chain cochain complexes, we define  internal $\Hom$s
$\cHom_A(M,N)$  by
\[
 \cHom_A(M,N)^i_j=  \Hom_{A^{\#}_{\#}}(M^{\#}_{\#},N^{\#[i]}_{\#[j]}),
\]
with differentials  $\pd f:= \pd_N \circ f \pm f \circ \pd_M$ and  $\delta f:= \delta_N \circ f \pm f \circ \delta_M$,
where $V^{\#}_{\#}$ denotes the bigraded vector space underlying a chain cochain complex $V$. 

We then define the  $\Hom$ complex $\hat{\HHom}_A(M,N)$ by
\[
 \hat{\HHom}_A(M,N):= \hat{\Tot} \cHom_A(M,N).
\]
\end{definition}

Observe that there is a multiplication $\hat{\HHom}_A(M,N)\ten \hat{\HHom}_A(N,P)\to \hat{\HHom}_A(M,P)$ --- the same is not true for $\Tot^{\Pi} \cHom_A(M,N)$ in general.

\begin{lemma}\label{denormmod}
 For $A \in DG^+dg\CAlg(\Q)$, the denormalisation functor $D \co DG^+dg\Mod(A) \to cdg\Mod(DA)$ to the category of $DA$-modules in cosimplicial chain complexes has a left adjoint $D^*_{\Mod}$. If, for $M \in cdg\Mod(DA)$, the underlying almost cosimplicial (see Definition \ref{almostdef}) graded module $M^{\#}_{\#}$ is isomorphic to $D^{\#}A_{\#}\ten_{A^0_{\#}}L_{\#}$ for an almost cosimplicial graded $A^0_{\#}$-module $L$, then 
\[
 D^*_{\Mod}(M)^{\#}_{\#} \cong A^{\#}_{\#}\ten_{A^0_{\#}}N_cL_{\#}
\]
as graded $A^{\#}_{\#}$-modules, where $N_c$ denotes Dold--Kan conormalisation.
\end{lemma}
\begin{proof}
Existence of the functor essentially follows because the functor $D$ preserves all limits. The description for objects of the form  $D^{\#}A_{\#}\ten_{A^0_{\#}}L_{\#}$ follows immediately from the  Dold--Kan correspondence for almost cosimplicial $A^0$-modules.  
\end{proof}
In fact, $D^*_{\Mod}(M)$ is always a quotient of $N_cM$ with $A$-module structure defined by the Alexander--Whitney cup product. 
For Cartesian $DA$-modules $M$, Lemma \ref{denormmod} (taking $L=M^0$) then implies that $D^*_{\Mod}M \cong N_cM$, with $DD^*_{\Mod}M \cong M$.

We now write $D \co DGdg\Mod(A) \to cdg\Mod(DA)$ for the composition of $D$ with the brutal truncation functor $\sigma^{\ge 0}$. We call a $DA$-module $M$ levelwise cofibrant if each $M^i$ is cofibrant in $dg\Mod(D^iA)$.  $\oL$ in $\oL D^*$. 
\begin{lemma}\label{Homrepmod}
For $A \in DG^+dg\CAlg(\Q)$,  a  levelwise cofibrant $DA$-module $M $ in $cdg\Mod(DA)$, and  $P \in DGdg\Mod( A)$, there is a canonical quasi-isomorphism
\[
 \oR \HHom_{DA}(M, DP) \simeq \Tot^{\Pi}\sigma^{\ge 0}\cHom_{A}( D_{\Mod}^*M, P). 
\]
\end{lemma}
\begin{proof}
The adjunction of Lemma \ref{denormmod} gives us an isomorphism
\[
 \HHom_{DA}(M, DP) \cong \z^0\cHom_{A}(D_{\Mod}^*M, P). 
\]

There is a  model structure on $cdg\Mod(DA)$  in which cofibrations and and weak equivalences are defined levelwise in cosimplicial degrees. Similarly, there is a model structure on  $DGdg\Mod(A)$  in which  weak equivalences  are defined levelwise in cochain degrees, and a morphism $M \to N$ is a  cofibration if $M^{\#} \to N^{\#}$ has the left lifting property with respect to all surjections of $A^{\#}$-modules in graded chain complexes. These are essentially the injective diagram model structures (somewhat confusingly, the diagrams take values in the projective model structure on chain complexes). 
The functors $D_{\Mod}^* \dashv D$ then form a Quillen adjunction, so to calculate $\oR \HHom_{DA}$ it suffices to take a fibrant replacement for $P$. 

Write $S$ for the chain cochain complex $S^{\#}_{\#} =\bigoplus_{i\ge 0} (\Q^{[-i]}_{[-i]} \oplus \Q^{[1-i]}_{[-i]})$, with all possible differentials in $S$ being the identity. Thus $S$ looks like a staircase in the second quadrant, with all columns acyclic. A fibrant replacement of $P$ is given by $\cHom_{\Q}(S,P)$, so 
\begin{align*}
 \oR \HHom_{DA}(M, DP) &\simeq \HHom_{A}(M, D \cHom_{\Q}(S,P))\\
&\cong \z^0\cHom_A(D^*_{\Mod}M, D \cHom_{\Q}(S,P))\\
&\cong \z^0\cHom_{\Q}(S, \cHom_{A}(D^*_{\Mod}M, P))\\
&\cong \Tot^{\Pi} \sigma^{\ge 0}  \cHom_{A}(D^*_{\Mod}M, P).
\end{align*}
\end{proof}

\begin{definition} 
 Given a stacky CDGA $A$, say that an $A$-module $M$ in chain cochain complexes is homotopy-Cartesian if the maps
\[
 A^i\ten^{\oL}_{A^0}M^0 \to M^i
\]
 are quasi-isomorphisms for all $i$.
\end{definition}

\subsection{Comparing DG Artin hypergroupoids and stacky CDGAs}\label{cfArtinstackysn}

Given a DG Artin $N$-hypergroupoid $X$, and an $O(X)$-module $M$, we may pull $M$ back along the unit $\eta \co O(X) \to DD^*O(X)$ of the adjunction $D^*\dashv D$, and then apply Lemma \ref{denormmod} to obtain a $D^*O(X)$-module $D^*_{\Mod}\eta^*M$.

\begin{definition}\label{almostdef}
 As for instance in \cite[Definition \ref{stacks2-delta*}]{stacks2}, define almost cosimplicial diagrams to be functors on the subcategory  $\Delta_*$  of the ordinal number category $\Delta$ containing only  those morphisms $f$ with $f(0)=0$; define almost simplicial diagrams dually. Thus an almost simplicial diagram $X_*$ in $\C$ consists of objects $X_n \in \C$, with all of the operations $\pd_i, \sigma_i$ of a simplicial diagram except $\pd_0$,  satisfying the usual relations. 
\end{definition}

Given a simplicial (resp. cosimplicial) diagram $X$, we write $X_{\#}$ (resp. $X^{\#}$) for the underlying almost simplicial diagram (resp. almost cosimplicial) diagram.

The denormalisation functor $D$ descends to a functor from graded-commutative algebras to almost cosimplicial algebras, with $D^*$ thus descending to a functor in the opposite direction. In other words, $(D^*B)^{\#}$ does not depend on $\pd^0_B$, and  $\pd^0_{DA}$ is the only part of the structure on  $DA$ to depend on $\pd_A$. The same is true for $D^*_{\Mod}$ applied to cosimplicial $DA$-modules. 

From this, it can be seen that for any DG Artin $N$-hypergroupoid $X$, the graded-commutative algebra $\H_0D^*O(X)^{\#}$ is freely generated over $\H_0O(X)^0= \H_0D^*O(X)^0$ by a graded projective module, and that 
\[
 \H_i^*O(X)^{\#}\cong \H_0D^*O(X)^{\#}\ten_{\H_0O(X)^0}\H_iO(X)^0. 
\]

If $M$ is a homotopy-Cartesian $O(X)$-module, the map 
\[
 (\oL\eta^*M)^0\ten_{D^*O(X)^0}^{\oL}D^{\#}D^*O(X) \to (\oL\eta^*M)^{\#} 
\]
 of almost cosimplicial chain complexes is a levelwise quasi-isomorphism. 
Applying $\oL D^*_{\Mod}$, it then follows from Lemma \ref{denormmod} 
that the map
$(\oL\eta^*M)^0\ten_{D^*O(X)^0}^{\oL} D^*O(X)^{\#} \to \oL D^*_{\Mod}(\eta^*M)^{\#}$ is a levelwise quasi-isomorphism, so $\oL D^*_{\Mod}(\eta^*M)^{\#}$ is also homotopy-Cartesian. 

\subsubsection{Equivalences of hypersheaves}\label{replaceresn}

Given a simplicial presheaf $F$ on $DG\Aff(R)$, there is an induced simplicial presheaf $FD$ on $DGdg\Alg(R)$ given by 
\[
 (FD)(A):= \ho\Lim_{i\in \Delta} F(D^iA).
\]
In the case of a simplicial derived affine $X$, regarded as a functor from $dg\Alg(R)$ to simplicial sets, $XD$ is represented by the cosimplicial chain CDGA $D^*O(X^{\Delta^{\bt}})$, given in cosimplicial level $m$ by $D^*O(X^{\Delta^m})$. Applying $D$ then gives a 
bisimplicial derived affine $\Spec DD^*O(X^{\Delta^{\bt}})$, and it is natural to consider the diagonal simplicial object
 $\diag \Spec DD^*O(X^{\Delta^{\bt}})$.

\begin{proposition}\label{replaceprop}
 For any Reedy fibrant simplicial derived affine $X$, the morphism
\[
 X \to  \diag \Spec DD^*O(X^{\Delta^{\bt}})
\]
of simplicial derived affines,
coming from the  isomorphisms $ D^*O(X^{\Delta^m})^0 = O(X_m)$, induces  a weak equivalence on the associated simplicial presheaves on $DG\Aff(R)$. 
\end{proposition}
\begin{proof}
We can rewrite the morphism above as
\[
 \ho \LLim_{i\in \Delta} X_i \to  \holim_{\substack{\lra \\ i,j\in \Delta\by \Delta}} \Spec D^jD^*O(X^{\Delta^i}).
\]
It will therefore suffice to show that for all $j$, the maps $\eta^j \co X \to \Spec D^jD^*O(X^{\Delta^{\bt}})$ give weak equivalences of the associated simplicial presheaves. For this it is enough to show that each $\eta^j$  is a  simplicial deformation retract.  

Now, the bisimplicial derived affine $j \mapsto X^{\Delta^j}$ admits a map from the bisimplicial object $j \mapsto X$, which we denote by $cX$. Forgetting the coface maps $\pd^0 \co \Delta^j \to \Delta^{j+1}$ in $\Delta^{\bt}$  gives us an almost cosimplicial simplicial set $\Delta^{\#}$, and hence an almost simplicial  simplicial derived affine $X^{\Delta^{\#}} $. Inclusion of the $0$th vertex makes $\Delta^0$ a deformation retract of $\Delta^j$, with obvious contracting homotopy $\Delta^1 \by \Delta^j \to \Delta^j$. These homotopies combine to give a homotopy $\Delta^1 \by \Delta^{\#} \to \Delta^{\#}$ of almost cosimplicial simplicial sets.

This homotopy makes $O(cX_{\#})$ a cosimplicial deformation retract of $O(X^{\Delta^{\#}})$ (as cosimplicial almost cosimplicial rings). 
Since $DD^*$ descends to a functor on almost cosimplicial algebras, applying $DD^*$ then makes $O(cX_{\#}) $ a cosimplicial deformation retract of $DD^*O(X^{\Delta^{\bt}})^{\#}$, so in particular the maps  $\eta^j \co X \to \Spec D^jD^*O(X^{\Delta^{\bt}})$  are all simplicial deformation retracts.
\end{proof}
If $X$ is a DG Artin $N$-hypergroupoid, with associated $N$-stack $\fX$, this means that 
 we can  recover $\fX$ from  the cosimplicial stacky CDGA $ D^*O(X^{\Delta^{\bt}})$, just as well as from the cosimplicial CDGA $O(X)$.

Proposition \ref{replaceprop} has the following immediate consequence:
\begin{corollary}\label{gooddescent}
For any simplicial presheaf  $F$ on $DG\Aff(R)$ and any Reedy fibrant simplicial derived affine $X$, there is a canonical weak equivalence 
\[
 \ho \Lim_{j\in\Delta} \map( \Spec DD^*O(X^{\Delta^j}), F) \to \map (X, F).
\]
\end{corollary}

\subsubsection{Tangent and cotangent complexes}\label{Artintgtsn}

Now consider a DG Artin $N$-hypergroupoid $X$ over $R$. Combining the adjunction $D^* \dashv D$ with Lemma \ref{denormmod} 
 and the universal property of K\"ahler differentials, observe that there is an isomorphism
\[
 D^*_{\Mod}(\Omega^1_{O(X)/R}\ten_{O(X)} DD^*O(X)) \cong  \Omega^1_{D^*O(X)/R}.
\]

Since $O(X)$ is Reedy cofibrant, Lemma \ref{Dstarlemma} ensures that $D^*O(X)$ is cofibrant in the model structure of Lemma \ref{bicdgamodel}. If we write $\cone_h$ and $\cone_v$ for cones in the chain and cochain directions respectively, then for any 
$D^*O(X)$-module $M$, the map $\cone_h\cone_v(M)^{[m-1]}_{[i]} \onto \cone_h(M)^{[m]}_{[i]}$ is levelwise acyclic, so cofibrancy of $D^*O(X)$ implies lifting of the respective spaces of derivations, or equivalently  surjectivity of
\[
 \cHom_{D^*O(X)}(\Omega^1_{D^*O(X)}, M)^{m-1}_{i} \onto \z^m\cHom_{D^*O(X)}(\Omega^1_{D^*O(X)}, M)_{i}.
\]
In other words, the columns of $\cHom_{D^*O(X)}(\Omega^1_{D^*O(X)}, M) $ are acyclic.

Because $X$ is an Artin $N$-hypergroupoid, there are in fact other restrictions on $\Omega^1_{D^*O(X)}$. The  results below (or an argument adapted from \cite[Lemma \ref{stacks2-truncate2}]{stacks2}) show that for $M \in DG^+dg\Mod({D^*O(X)})$, the rows $ \cHom_{D^*O(X)}(\Omega^1_{D^*O(X)}, M)^i$ are acyclic for $i<-N$. Moreover, the argument of \cite[Corollary \ref{ddt1-cohowelldfn}]{ddt1} shows that when  $M$ is concentrated in degree $(0,0)$, we have $\H_j\cHom_{D^*O(X)}(\Omega^1_{D^*O(X)}, M)^i=0$ for $i,j \ne 0$, although we will not need this.

\begin{lemma}\label{contractlemma}
 If, for  $A\in DG^+dg\CAlg(R)$,  an $A$-module $M \in DG^+dg\Mod(R)$  admits an $A$-linear contracting homotopy, 
then the map
\[
\ho\LLim_{i\in \Delta} \Spec D^iA \to \ho\LLim_{i\in \Delta} \Spec D^i(A \oplus M)
\]
of simplicial presheaves on $DG^+\Aff(R)$ is a weak equivalence. 
\end{lemma}
\begin{proof}
 A contracting homotopy is the same as a section of the canonical  map $\cocone_v(M) \to M$ of $A$-modules, where the vertical cocone $\cocone_v(M)$ is given by $M \ten \cocone_v(\Q)$ for $\cocone_v(\Q)= \Q \oplus \Q^{[-1]}$ with $\pd$ the identity.

On applying $D$, we then have a section $s$ of $ D\cocone_v(M) \to DM$ as $DA$-modules in cosimplicial chain complexes. The Alexander--Whitney cup product gives a canonical morphism  $D\cocone_v(M) \to (DM)\ten D\cocone_v(\Q)$.   Since $\Q^{\Delta^1} = \Q \oplus  D\cocone_v(\Q)$, our section $s$ then gives a contracting homotopy $DM \to (DM)^{\Delta^1}$, so evaluation at $0 \in \Delta^1$ is the zero map, and evaluation at $1 \in \Delta^1$ the identity.

We therefore have a homotopy $\Delta^1 \by \Spec D(A \oplus M) \to \Spec D(A \oplus M)$ of simplicial derived affine schemes (and hence of simplicial presheaves) realising $\Spec DA$ as a deformation retract of $\Spec D(A \oplus M)$. 
\end{proof}

From now on, we will simply write $(\Spec DA)^{\sharp}:= \ho\LLim_{i\in \Delta} (\Spec D^iA)^{\sharp}$ for the  \'etale hypersheafification of the simplicial derived affine $\Spec DA$.

\begin{corollary}\label{suspendcor}
 For  $A\in DG^+dg\CAlg(R)$,   $M \in DG^+dg\Mod(A)$ and $m> 0$,
there is a weak equivalence
\[
 (\Spec D(A \oplus M^{[-m]}))^{\sharp}\simeq (\Spec D(A \oplus M) \by S^m)^{\sharp}\cup^{\oL}_{\Spec D(A)^{\sharp} \by S^m} (\Spec DA)^{\sharp}
\]
of \'etale hypersheaves on $DG\Aff(R)$, where $S^m \simeq \Delta^m/\pd \Delta^m$ is the $m$-sphere.
\end{corollary}
\begin{proof}
 We have an exact sequence $0 \to M[-1] \to \cocone_v(M) \to M \to 0$. Since $\cocone_v(M)$ has a contracting homotopy, Lemma \ref{contractlemma} shows that $(\Spec D(A \oplus \cocone_v(M))^{\sharp} \simeq (\Spec DA)^{\sharp}$. Thus $ (\Spec D(A \oplus M[-1]))^{\sharp}$ is the homotopy pushout of the nilpotent maps
\[
 (\Spec DA)^{\sharp} \la (\Spec D(A \oplus M))^{\sharp} \to (\Spec DA)^{\sharp},
\]
so is the suspension of $(\Spec D(A \oplus M))^{\sharp}$ over $ (\Spec DA)^{\sharp}$. Replacing $M$ with $M[-j]$ for $0 \le j \le m$ then gives the desired result by induction.
\end{proof}

\begin{definition}
 For any stacky CDGA $A$ over $R$, the module  $\Omega^1_{A/R}$ of K\"ahler differentials is  an $A$-module in chain cochain complexes, and we define $\Omega^p_{A/R}:= \Lambda_A^p \Omega^1_{A/R}$,  denoting its differentials (inherited from $A$) by $(\pd,\delta)$. There is also a de Rham differential $d \co \Omega^p_{A/R} \to \Omega^{p+1}_{A/R}$.
\end{definition}

\begin{proposition}\label{tgtcor1}
 For a DG Artin $N$-hypergroupoid $X$ over $R$, 
and any $M \in DG^+dg\Mod(D^*O(X))$,   the cotangent complex $\bL^{X/R}$ satisfies 
\begin{align*}
 \oR\HHom_{O(X)}(\bL^{X/R}, DM) &\simeq \Tot^{\Pi}\sigma^{\ge -N}\cHom_{D^*O(X)}(\Omega^1_{D^*O(X)/R}, M)\\
&\simeq \hat{\HHom}_{D^*O(X)}(\Omega^1_{D^*O(X)/R}, M).
\end{align*}
The same is true for any homotopy-Cartesian module $M \in DGdg\Mod(D^*O(X))$.
\end{proposition}
\begin{proof}
If $\fX= X^{\sharp}$ is the derived $N$-stack associated to $X$, then for $f \co X \to \fX$, Lemma \ref{suspendcor} gives an equivalence
  \[
 \oR\HHom_{\O_{\fX}}(\oL c^* \bL^{\fX^{S^N}}, \oR f_* DM) \simeq \oR\HHom_{O(X)}(\Omega^1_{X/R}, D(M^{[-N]})),
\]
for the map $c \co \fX \to \fX^{S^N}$. Since $\fX$ is an $N$-stack, there is a canonical equivalence $\bL^{\fX} \simeq  \bL^{\fX^{S^N}}_{[N]}$ (cf. \cite[Corollary \ref{stacks2-loopcot}]{stacks2}), so Lemma \ref{Homrepmod} gives 
\[
 \oR\HHom_{O(X)}(\bL^{X/R}, DM) \simeq \Tot^{\Pi}\sigma^{\ge -N}\cHom_{D^*O(X)}(\Omega^1_{D^*O(X)/R}, M).
\]

Since $X$ is \emph{a fortiori} an $(N+r)$-hypergroupoid for all $r\ge 0$, we may replace $\sigma^{\ge -N}$ with $\sigma^{\ge -N-r}$. Taking the filtered colimit over all $r$ then replaces $\Tot^{\Pi}\sigma^{\ge -N}\cHom$ with $\hat{\Tot}\cHom= \hat{\HHom}$, yielding the second quasi-isomorphism. Finally, observe that for $M$ homotopy-Cartesian, we must have $M^i$ acyclic for $i<0$, because $A^i=0$. Thus we can replace $M$ with $\sigma^{\ge 0}M$, and the results apply.

\end{proof}

\begin{proposition}\label{tgtcor2}
 For a DG Artin $N$-hypergroupoid $X$ over $R$, 
any $M \in DG^+dg\Mod(D^*O(X))$, and any $r \ge 0$, we have
\[
 \oR\HHom_{O(X)}((\bL^{X/R})^{\ten r}, DM) \simeq \hat{\HHom}_{D^*O(X)}((\Omega^1_{D^*O(X)/R})^{\ten r}, M).
\]
\end{proposition}
\begin{proof}
There is a natural map $\bL^{X/R} \to \Omega^1_{X/R}$ coming from Definition \ref{cotdef}. Corollary \ref{Homrepmod} thus ensures that we have a natural map from right to left. 

First consider the case where $M=M^i[-i]$ is an $O(X_0)$-module, since  $O(X)^0=D^*O(X)^0$.
 Writing $L$ for the pullback of $\bL^{X/R}$ to $X_0$, and $\bar{\Omega}$ for the pullback of $\Omega^1_{D^*O(X)/R} $ to $O(X_0)$, we therefore wish to show that the map
\[
 \hat{\Tot} \cHom_{O(X_0)}(\bar{\Omega}^{\ten r}, M)   \to  \oR\HHom_{O(X_0)}(L^{\ten r}, M)
\]
is a quasi-isomorphism.

Proposition \ref{tgtcor1} implies that $\bar{\Omega}$ is levelwise quasi-isomorphic to the brutal cotruncation $\sigma^{\le N}\bar{\Omega}:= \bar{\Omega}/\sigma^{\ge N+1}$, so we may replace   $\bar{\Omega}^{\ten r} $ with a bicomplex in cochain degrees $[0, Nr]$. The resulting $\cHom$ bicomplex is bounded below, so $\hat{\Tot}$ is equivalent to $\Tot^{\Pi}$, and we need only show a quasi-isomorphism
\[
 \oR\HHom_{O(X_0)}( (\Tot \sigma^{\le N}\bar{\Omega})^{\ten r}, M)   \simeq  \oR\HHom_{O(X_0)}(L^{\ten r}, M),
\]
 which follows because $L \simeq \Tot \sigma^{\le N}\bar{\Omega}$ by Proposition \ref{tgtcor1}. (We could also take the limit over all $N$ to give $ L \simeq \Tot^{\Pi}\bar{\Omega}$.)

For  general $M$, we have 
\[
\cHom_{D^*O(X)}((\Omega^1_{D^*O(X)/R})^{\ten r}, M) =\Lim_i \cHom_{D^*O(X)}((\Omega^1_{D^*O(X)/R})^{\ten r}, \sigma^{\le i} M),
\]
and the spectral sequence for the cochain brutal truncation filtration for $M$ combines with boundedness of $\bar{\Omega}^{\ten r} $ to show that the right-hand complex is bounded below in cochain degrees, allowing us to replace $\hat{\HHom}$ with $\Tot^{\Pi}\sigma^{\ge -Nr}\cHom$. The spectral sequence therefore satisfies complete convergence, and the general quasi-isomorphism follows from the quasi-isomorphism above for the graded pieces $M^i[-i]$.
\end{proof}

%

\subsection{Poisson and symplectic structures}\label{bipoisssn}

To allow for some flexibility in computations and resolutions, we will not just consider stacky CDGAs of the form $D^*O(X)$ for DG Artin $N$-hypergroupoids $X$, but will fix  $A \in DG^+dg\CAlg(R)$ satisfying:
\begin{enumerate}
 \item for any cofibrant replacement $\tilde{A}\to A$ in the model structure of Lemma \ref{bicdgamodel}, the morphism $\Omega^1_{\tilde{A}/R}\to \Omega^1_{A/R}$ is a levelwise quasi-isomorphism,
\item  the $A^{\#}$-module $(\Omega^1_{A/R})^{\#}$ in   graded chain complexes is cofibrant (i.e. it has the left lifting property with respect to all surjections of $A^{\#}$-modules in graded chain complexes),
\item there exists $N$ for which the chain complexes $(\Omega^1_{A/R}\ten_AA^0)^i $ are acyclic for all $i >N$.
\end{enumerate}

Observe that these conditions are satisfied by $D^*O(X)$, by the results above, but they are also satisfied by more general stacky CDGAs. The first two conditions do not require $A$ to be cofibrant in the model structure of Lemma \ref{bicdgamodel}; it is enough for $A^{\#}$ to be cofibrant as a graded chain CDGA over $R$, or for $A \in DG^+dg_+\CAlg(R)$ with  $A^{\#}_{\#}$  ind-smooth as a bigraded commutative algebra over $R_{\#}$.

\subsubsection{Polyvectors}

All the definitions and properties of \S \ref{affinesn}   now carry over. In particular: 

\begin{definition}\label{bipoldef}
Given a stacky CDGA $A$ over $R$ as above, define the complex of $n$-shifted polyvector fields  on $A$ by
\[
 \widehat{\Pol}(A,n):= \prod_{j \ge 0} \hat{\HHom}_A(\CoS_A^j((\Omega^1_{A/R})_{[-n-1]}),A). 
\]
This has a filtration  by complexes
\[
F^p\widehat{\Pol}(A,n):= \prod_{j \ge p} \hat{\HHom}_A(\CoS_A^j((\Omega^1_{A/R})_{[-n-1]}),A),
\]
with $[F^i,F^j] \subset F^{i+j-1}$ and $F^i F^j \subset F^{i+j}$, where the commutative product and Schouten--Nijenhuis bracket are defined as before.
\end{definition}

We now define the space $\cP(A,n)$ of Poisson structures and its tangent space $T\cP(A,n)$ by the formulae of Definitions \ref{poissdef}, \ref{Tpoissdef}. As in Definition \ref{sigmadef}, there is a  canonical tangent vector $\sigma \co \cP(A,n) \to  T\cP(A,n)$. 

As well as composition for internal $\cHom$'s, we have maps $\cHom(M_1,P_1)\ten \cHom(M_2, P_2) \to \cHom(M_1\ten M_2, P_1\ten P_2)$, and hence
\begin{align*}
  \cHom_A(\Omega_A^1, \cHom_A(M, A))^{\ten p} &\to \cHom_A((\Omega_A^1)^{\ten p}, \cHom_A(M, A)^{\ten p})\\
& \to \cHom_A((\Omega_A^1)^{\ten p}, \cHom_A((M^{\ten p}, A)).
\end{align*}

Substituting $M= \bigoplus_{j=2}^{r-1} \CoS_A^j((\Omega^1_{A/R})_{[-n-1]})$, 
taking shifts and  $S_p$-coinvariants, and applying $\hat{\Tot}$, we see that an element
\[
 \pi \in F^2\widehat{\Pol}(A,n)^{n+2}/F^r   
\]
thus defines a contraction morphism
\[
\mu(-, \pi) \co \Tot^{\Pi} \Omega_A^p \to F^p\widehat{\Pol}(A,n)^{n+2}/F^{p(r-1)}
\]
of bigraded vector spaces, noting that $\Tot^{\Pi} \Omega_A^p = \hat{\Tot}\Omega_A^p $.
When $r=3$ and $p=1$, we simply denote this morphism  by $\pi_2^{\sharp}\co \Tot^{\Pi} \Omega_A^1 \to \hat{\HHom}_A(\Omega^1_A, A)[-n]$.

\begin{definition}\label{binondegdef}
We say that a Poisson structure $\pi \in  \cP(A,n)/F^p$ is non-degenerate if $\Tot^{\Pi} (\Omega_{A/R}^1\ten_AA^0)$ is a perfect complex over $A^0$ and the map
\[
 \pi_2^{\sharp}\co \Tot^{\Pi} (\Omega_{A/R}^1\ten_AA^0) \to \hat{\HHom}_A(\Omega^1_A, A^0)[-n]
\]
is a quasi-isomorphism.  
\end{definition}

\begin{lemma}\label{binondeglemma}
 If $\pi_2 \in  \cP(A,n)/F^3$ is non-degenerate, then the maps
\[
 \mu(-, \pi_2) \co \Tot^{\Pi}(\Omega_A^p\ten_AM) \to \hat{\HHom}_A(\CoS_A^p((\Omega^1_{A/R})_{[-n-1]}),M)
\]
are  quasi-isomorphisms for all $M \in DG^+dg\Mod(A)$.
\end{lemma}
\begin{proof}
We may proceed as in Proposition \ref{tgtcor2}. By hypothesis, $ (\Omega_{A/R}^1\ten_AA^0)^i$ is acyclic for all $i>N$, so $ \cHom_A(\CoS_A^p((\Omega^1_{A/R})_{[-n-1]}),A)^i$ is acyclic for all $i<-Np$, meaning that we may replace $\hat{\HHom}$ with $\Tot^{\Pi}\cHom$. Since $M= \Lim_i (M/\sigma^{\ge i}M)$, a spectral sequence argument allows us to reduce to the case where $M \in dg\Mod(A^0)$. 

We may also replace $ (\Omega_{A/R}^1\ten_AA^0)$ with its brutal truncation $ \sigma^{\le N}(\Omega_{A/R}^1\ten_AA^0)$, so all complexes are bounded in the cochain direction,  and then observe that non-degeneracy of $\pi$ gives us  quasi-isomorphisms
\[
  (\Tot \sigma^{\le N}(\Omega_{A/R}^1\ten_AA^0))^{\ten_{(A^0)}^p}\ten_{A^0}M \to \HHom_{A^0}( \sigma^{\le N}(\Omega_{A/R}^1\ten_AA^0)^{\ten_{(A^0)}^p}, M)[-np];
\]
the required result follows on taking $S_p$-coinvariants.
\end{proof}

\subsubsection{The de Rham complex}\label{biprespsn}

\begin{definition}\label{biDRdef}
Define the de Rham complex $\DR(A/R)$ to be the product total complex of the double cochain complex
\[
 \Tot^{\Pi} A \xra{d} \Tot^{\Pi}\Omega^1_{A/R} \xra{d} \Tot^{\Pi}\Omega^2_{A/R}\xra{d} \ldots,
\]
so the total differential is $d \pm \delta \pm \pd$.

We define the Hodge filtration $F$ on  $\DR(A/R)$ by setting $F^p\DR(A/R) \subset \DR(A/R)$ to consist of terms $\Tot^{\Pi}\Omega^i_{A/R}$ with $i \ge p$.
\end{definition}

The definitions of shifted symplectic structures from \S \ref{prespsn} now carry over:

\begin{definition}\label{biPreSpdef}
 Define the space $\PreSp(A,n)$ of $n$-shifted pre-symplectic structures on $A/R$ by  writing $\PreSp(A,n)/F^{i+2}:= \mmc(F^2  \DR(A/R)[n+1]/F^{i+2})$, and setting $ \PreSp(A,n):= \Lim_i \PreSp(A,n)/F^{i+2}$.

Say that an $n$-shifted pre-symplectic structure $\omega$ is symplectic if 
 $\Tot^{\Pi} (\Omega_{A/R}^1\ten_AA^0)$ is a perfect complex over $A^0$ and the map
\[
\omega_2^{\sharp}\co \hat{\HHom}_A(\Omega^1_A, A^0)[-n]\to  \Tot^{\Pi} (\Omega_{A/R}^1\ten_AA^0) 
\]
is a quasi-isomorphism.  
Let $\Sp(A,n) \subset \PreSp(A,n)$  consist of the symplectic structures --- this is a union of path-components.
\end{definition}

\begin{lemma}
 For a DG Artin $N$-hypergroupoid $X$, the space of closed $p$-forms of degree $n$ from \cite{PTVV} is given by
\[
 \cA^{p,cl}_R(X^{\sharp},n)\simeq\ho\Lim_{j\in\Delta}\mmc(F^p  \DR(D^*O(X^{\Delta^j})/R)^{[n+p-1]}).
\]
\end{lemma}
\begin{proof}
 This follows by combining Corollary \ref{gooddescent} and Proposition \ref{tgtcor2}. 
\end{proof}

\subsubsection{Compatible pairs}\label{Artincompat}

The proof of Lemma \ref{mulemma} adapts to give  maps
\[
 (\pr_2 + \mu\eps) \co  \PreSp(A,n)/F^{p} \by \cP(A,n)/F^{p} \to T\cP(A,n)/F^{p}
\]
 over $\cP(A,n)/F^{p}$ for all $p$, compatible with each other.

We may now define the space $\Comp(A, n)$ of compatible pairs as in Definition \ref{compdef}, with the results of \S \ref{affinesn} all adapting to show that the maps     $\Sp(A,n) \la  \Comp(A,n)^{\nondeg} \to  \cP(A,n)^{\nondeg}$ are weak equivalences. 

\subsection{Diagrams and functoriality}\label{Artindiagramsn}

We now extend the constructions of \S \ref{DMdiagramsn} to stacky CDGAs.

\subsubsection{Definitions}

Given a small category $I$, an $I$-diagram $A$ of stacky CDGAs over $R$, and   $A$-modules $M,N$ in $I$-diagrams of chain cochain complexes, we can define the  cochain complex $\hat{\HHom}_A(M,N)$ to be the equaliser of the obvious diagram
\[
\prod_{i\in I} \hat{\HHom}_{A(i)}(M(i),N(i)) \implies \prod_{f\co i \to j \text{ in } I}   \hat{\HHom}_{A(i)}(M(i),f_*N(j)).
\]
All the constructions of \S \ref{bipoisssn} then adapt immediately; in particular, we can define 
\[
 \widehat{\Pol}(A,n):= \prod_{j \ge 0} \hat{\HHom}_A(\CoS_A^j((\Omega^1_{A/R})_{[-n-1]}),A),
\]
leading to a space $\cP(A,n)$ of Poisson structures.

We can now adapt the formulae of \S \ref{DMdiagramsn} to this setting, defining  pre-symplectic structures by 
\[
 \PreSp(A,n):=  \PreSp(A(0),n)= \Lim_{i\in [m]}  \PreSp(A(i),n)
\]
for any $[m]$-diagram $A$ of stacky CDGAs over $R$, and setting $ \Comp(A,n)$ be the homotopy vanishing locus of the obvious maps
\[
(\mu - \sigma) \co  \PreSp(A,-1) \by \cP(A,n) \to  T\cP(A,n).
\]
over $\cP(A,n)$.

The obstruction functors and their towers from \S \ref{towersn} also adapt immediately, giving the obvious analogues of the obstruction spaces defined in terms of 
\[
\hat{\HHom}_A(\CoS^p_A(\Omega^1_{A/R}[n+1]),A), \quad \Tot^{\Pi}\Omega^p_{A(0)}.
\]

\subsubsection{Functors and descent}\label{bidescentsn}


For $[m]$-diagrams in $DG^+dg\CAlg(R)$, we will consider the injective model structure, so an $[m]$-diagram $A$ is cofibrant if each $A(i)$ is cofibrant for the model structure of Lemma \ref{bicdgamodel}, and is fibrant if the maps $A(i) \to A(i+1)$ are all surjective.

\begin{lemma}\label{bicalcTlemma2}
If $D=(A\to B)$ is a fibrant cofibrant  $[1]$-diagram in $DG^+dg\CAlg(R)$ which is formally \'etale in the sense that the map
\[
 \{\Tot \sigma^{\le q} (\Omega_{A}^1\ten_{A}B^0)\}_q \to \{\Tot \sigma^{\le q}(\Omega_{B}^1\ten_BB^0)\}_q
\]
is a pro-quasi-isomorphism, then the map    
\[
 \hat{\HHom}_D(\CoS_D^k\Omega^1_D,D) \to \hat{\HHom}_{A}(\CoS_{A}^k\Omega^1_{A},A),
\]
is a quasi-isomorphism for all $k$.
\end{lemma}
\begin{proof}
 This follows by reasoning as in the proof of Proposition \ref{tgtcor2}.
\end{proof}

\begin{definition}
 Write  $DG^+dg\CAlg(R)_{c, \onto}\subset DG^+dg\CAlg(R) $ for the subcategory with all cofibrant stacky CDGAs (in the model structure of Lemma \ref{bicdgamodel}) over $R$ as objects, and only surjections as morphisms.
\end{definition}

For the notion of being formally \'etale from Lemma \ref{bicalcTlemma2}, we may extend the conditions of Properties \ref{Fproperties} to constructions on $DG^+dg\CAlg(R)_{c, \onto}$, with quasi-isomorphisms taken levelwise. The constructions  $\cP(-,n)$, $\Comp(-,n)$, $\PreSp(-,n)$,  and their associated filtered and graded functors, all satisfy these properties. The first two properties follow from the right lifting property for  fibrations (in the injective model structure on diagrams), and the third from Lemma \ref{bicalcTlemma2}.

Thus the simplicial classes $\coprod_{ A \in B_m DG^+dg\CAlg(R)_{c, \onto}} F(A)$ fit together to give a complete Segal space $\int F$ over the nerve $BDG^+dg\CAlg(R)_{c, \onto} $. 

Definition \ref{LintFdef} then adapts to give us an $\infty$-category $\oL\int F$,  and Definition \ref{inftyFdef} adapts to give
an $\infty$-functor
\[
 \oR F \co \oL DG^+dg\CAlg(R)^{\et} \to  \oL s\Set
\]
with $(\oR F)(A) \simeq F(A)$
for all cofibrant stacky CDGAs $A$ over $R$, where $DG^+dg\CAlg(R)^{\et} \subset DG^+dg\CAlg(R)$ is the $2$-sub-$\infty$-category of homotopy formally \'etale morphisms.

An immediate consequence of \S \ref{Artincompat} is that the canonical maps
\begin{eqnarray*}
  \oR\Comp(-,n)^{\nondeg} &\to& \oR\Sp(-,n)  \\   
    \oR\Comp(-,n)^{\nondeg} &\to&  \oR\cP(-,n)^{\nondeg}           
\end{eqnarray*}
 are weak equivalences of $\infty$-functors on the full subcategory of $\oL DG^+dg\CAlg(R)^{\et}$ consisting of objects satisfying the conditions of \S \ref{bipoisssn}. 

Corollary \ref{gooddescent} and Proposition \ref{tgtcor2} ensure that if a morphism $X \to Y$ of DG Artin $N$-hypergroupoids becomes an equivalence on hypersheafifying, then $D^*O(Y) \to D^*O(X)$ is formally \'etale in the sense of Lemma \ref{bicalcTlemma2}. In particular this means that the maps $\pd^i \co D^*O(X^{\Delta^j}) \to D^*O(X^{\Delta^{j+1}})$ and $\sigma^i \co   D^*O(X^{\Delta^{j+1}})\to D^*O(X^{\Delta^j})$ are formally \'etale. Thus $D^*O(X^{\Delta^{\bt}})$ can be thought of as a DM hypergroupoid in stacky CDGAs, and we may make the following definition:

\begin{definition}\label{biinftyFXdef}
Given a DG Artin $N$-hypergroupoid $X$ over $R$ and   any of the functors $F$ above, write 
\[
 F(X):= \ho\Lim_{j \in \Delta} \oR F(D^* O(X^{\Delta^j})).
\]
\end{definition}

\begin{proposition}\label{biinftyFXwell}
If $Y \to X$ is a trivial DG Artin hypergroupoid, then the morphism 
\[
 F(X) \to F(Y)
\]
is an equivalence for any of the constructions $F= \cP, \Comp, \PreSp$.
\end{proposition}
\begin{proof}
The proof of Proposition \ref{inftyFXwell} adapts, replacing  Propositions \ref{DRobs} and \ref{compatobs} with Corollary \ref{gooddescent} and Proposition  \ref{tgtcor2}. 
 \end{proof}

Thus the following is well-defined:
\begin{definition}\label{PdefArtin}
 Given a strongly quasi-compact DG Artin $N$-stack $\fX$,  define the spaces  $\cP(\fX,n)$, $\Comp(\fX,n)$, $\Sp(\fX,n)$  to be  the spaces
$
 \cP(X,n), \Comp(X,n), \Sp(X,n)
$
for any DG Artin $N$-hypergroupoid $X$ with $X^{\sharp} \simeq \fX$.
\end{definition}

\begin{examples}\label{2PBG}
If $R=\H_0R$, with $Y/R$ a smooth affine scheme  equipped with an action of a Lie algebra $\g$,  we start by considering $2$-shifted Poisson structures on the  the Chevalley--Eilenberg complex $O([Y/\g])$ of Example \ref{DstarBG} (a cochain CDGA). In this case, the DGLA $F^i\widehat{\Pol}( [Y/\g],2)[3]$ is concentrated in cochain degrees $[2i-3, \infty)$. In particular, this means that $\cP( [Y/\g],2) \cong \z^4(\gr_F^2\widehat{\Pol}( [Y/\g],2))$ is a discrete space (so all homotopy groups are trivial). Explicitly, 
\[
 \cP( [Y/\g],2) \cong \{ \pi \in (S^2\g \ten O(Y))^{\g} ~:~ [\pi, a]=0 \in \g \ten O(Y) ~\forall a \in O(Y)\}.
\]
Specialising to the case $Y=\Spec R$, we have  $\cP( B\g,2) \cong (S^2\g)^{\g}$, the set of quadratic Casimir elements. 

If $\g$ is the Lie algebra of a linear algebraic group $G$ as in Example \ref{DstarBG}, then for $X:=B[*/G]$,   the spaces $\cP(D^*O(X^{\Delta^j}),2)$  
are all discrete sets, so the space $\cP(BG,2)$ is just the equaliser of the maps $\cP( B\g,2) \implies \cP([G/\g^{\oplus 2}],2)  $ coming from the   vertex maps  
of the 
cosimplicial CDGA $D^*O(X^{\Delta^{\bt}})$  
given by
\[
 \xymatrix@1{ O(B\g) \ar@<1ex>[r] \ar@<-1ex>[r] & \ar@{.>}[l] O([G/\g^{\oplus 2}]) \ar[r] \ar@/^/@<0.5ex>[r] \ar@/_/@<-0.5ex>[r] & \ar@{.>}@<0.75ex>[l] \ar@{.>}@<-0.75ex>[l]   
O([G^2/\g^{\oplus 3}]) \ar@/^1pc/[rr] \ar@/_1pc/[rr] \ar@{}[rr]|{\cdot} \ar@{}@<1ex>[rr]|{\cdot} \ar@{}@<-1ex>[rr]|{\cdot} && O([G^3/\g^{\oplus 4}]){}   \ar@/^1.2pc/[rr] \ar@/_1.2pc/[rr]\ar@{}[rr]|{\cdot} \ar@{}@<1ex>[rr]|{\cdot} \ar@{}@<-1ex>[rr]|{\cdot}& & {\phantom{E}}\cdots .} 
\]
\vskip 1ex
These \'etale vertex maps are induced by applying $S^2$ to the maps
$
 \g \implies \g^{\oplus 2}\ten O(G)
$
which when evaluated on $g \in G$ send $v \in \g$ to $(v, gvg^{-1})$ and $(gvg^{-1},v)$ respectively. The equaliser $\cP(BG,2)\subset (S^2\g)^{\g}$  is thus 
\[
 \cP(BG,2) \cong (S^2\g)^G,
\]
with each element
corresponding to a cosimplicial $P_3$-algebra structure on  $D^*O((B[*/G])^{\Delta^{\bt}})$.
 Taking $G$  reductive and restricting to path components recovers an example in \cite[\S 3.1]{CPTVV}.
\end{examples}

\begin{remark}\label{cfCPTVV} To relate Definition \ref{PdefArtin} with the Poisson structures of \cite{CPTVV}, first note that reindexation gives an equivalence of categories between double complexes and
the ``graded mixed complexes'' of \cite{PTVV,CPTVV} (but beware \cite[Remark 1.1.2]{CPTVV}: in the latter paper,  ``graded mixed complexes''  do not have mixed differentials). Thus ``graded mixed cdgas''  are just stacky CDGAs, and to a derived stack $\fX$, \cite[Definition 4.2.11]{CPTVV} associates a sheaf $\bD_{\fX/\fX_{\dR}}$ of stacky CDGAs  on the de Rham stack $\fX_{\dR}$, defining Poisson structures in terms of polyvectors on $\bD_{\fX/\fX_{\dR}}$. A comparison with our definition should then involve the observation that  $\Spec DD^*O(X^{\Delta^j})$ is a model for the relative de Rham stack $(X_j/\fX)_{\dR}= (X_j)_{\dR}\by^h_{\fX_{\dR}}\fX$, possibly with $D^*O(X^{\Delta^j}) \simeq \bD_{X_j/\fX}$.
\end{remark}

Combined with  the results   above, an immediate consequence of \S \ref{Artincompat} is: 
  \begin{theorem}\label{Artinthm}
For any strongly quasi-compact DG Artin $N$-stack $\fX$ over $R$, there are natural weak equivalences
\[
 \Sp(\fX,n) \la \Comp(\fX,n)^{\nondeg}\to \cP(\fX,n)^{\nondeg}.
\]
\end{theorem}
Modulo the comparison suggested in Remark \ref{cfCPTVV}, an alternative proof of Theorem \ref{Artinthm} is given as \cite[Theorem 3.2.4]{CPTVV}.

\bibliographystyle{alphanum}
\bibliography{references.bib}

\newcommand{\etalchar}[1]{$^{#1}$}
\def\cprime{$'$}
\begin{thebibliography}{BBBBJ}

\bibitem[Aok]{aoki}
Masao Aoki.
\newblock Deformation theory of algebraic stacks.
\newblock {\em Compos. Math.}, 141(1):19--34, 2005.

\bibitem[BBBBJ]{BBBJdarboux}
O.~Ben-Bassat, C.~Brav, V.~Bussi, and D.~Joyce.
\newblock A ``{D}arboux theorem'' for shifted symplectic structures on derived
  {A}rtin stacks, with applications.
\newblock {\em Geom. Topol.}, 19:1287--1359, 2015.
\newblock arXiv:1312.0090 [math.AG].

\bibitem[BG]{BouazizGrojnowski}
E.~Bouaziz and I.~Grojnowski.
\newblock A {$d$}-shifted {D}arboux theorem.
\newblock arXiv:1309.2197v1 [math.AG], 2013.

\bibitem[CPT{\etalchar{+}}]{CPTVV}
D.~Calaque, T.~Pantev, B.~To{\"e}n, M.~Vaqui{\'e}, and G.~Vezzosi.
\newblock Shifted {P}oisson structures and deformation quantization.
\newblock {\em J. Topol.}, 10(2):483--584, 2017.
\newblock arXiv:1506.03699v4 [math.AG].

\bibitem[GJ]{sht}
Paul~G. Goerss and John~F. Jardine.
\newblock {\em Simplicial homotopy theory}, volume 174 of {\em Progress in
  Mathematics}.
\newblock Birkh{\"a}user Verlag, Basel, 1999.

\bibitem[Hin]{hinstack}
Vladimir Hinich.
\newblock D{G} coalgebras as formal stacks.
\newblock {\em J. Pure Appl. Algebra}, 162(2-3):209--250, 2001.

\bibitem[Hir]{Hirschhorn}
Philip~S. Hirschhorn.
\newblock {\em Model categories and their localizations}, volume~99 of {\em
  Mathematical Surveys and Monographs}.
\newblock American Mathematical Society, Providence, RI, 2003.

\bibitem[Hov]{hovey}
Mark Hovey.
\newblock {\em Model categories}, volume~63 of {\em Mathematical Surveys and
  Monographs}.
\newblock American Mathematical Society, Providence, RI, 1999.

\bibitem[KSM]{KosmannSchwarzbachMagriPN}
Yvette Kosmann-Schwarzbach and Franco Magri.
\newblock Poisson--{N}ijenhuis structures.
\newblock {\em Ann. Inst. H. Poincar\'e Phys. Th\'eor.}, 53(1):35--81, 1990.

\bibitem[KV]{KhudaverdianVoronov}
H.~M. Khudaverdian and Th.~Th. Voronov.
\newblock Higher {P}oisson brackets and differential forms.
\newblock In {\em Geometric methods in physics}, volume 1079 of {\em AIP Conf.
  Proc.}, pages 203--215. Amer. Inst. Phys., Melville, NY, 2008.
\newblock arXiv:0808.3406v2 [math-ph].

\bibitem[Lur]{lurie}
J.~Lurie.
\newblock {\em Derived Algebraic Geometry}.
\newblock PhD thesis, M.I.T., 2004.
\newblock www.math.harvard.edu/$\sim$lurie/papers/DAG.pdf or
  http://hdl.handle.net/1721.1/30144.

\bibitem[Mel]{melaniPoisson}
Valerio Melani.
\newblock Poisson bivectors and {P}oisson brackets on affine derived stacks.
\newblock {\em Adv. Math.}, 288:1097--1120, 2016.
\newblock arXiv:1409.1863v3 [math.AG].

\bibitem[Pri1]{ddt1}
J.~P. Pridham.
\newblock Unifying derived deformation theories.
\newblock {\em Adv. Math.}, 224(3):772--826, 2010.
\newblock arXiv:0705.0344v6 [math.AG], corrigendum 228 (2011), no. 4,
  2554--2556.

\bibitem[Pri2]{stacksintro}
J.~P. Pridham.
\newblock Notes characterising higher and derived stacks concretely.
\newblock arXiv:1105.4853v3 [math.AG], 2011.

\bibitem[Pri3]{stacks2}
J.~P. Pridham.
\newblock Presenting higher stacks as simplicial schemes.
\newblock {\em Adv. Math.}, 238:184--245, 2013.
\newblock arXiv:0905.4044v4 [math.AG].

\bibitem[Pri4]{DQLag}
J.~P. Pridham.
\newblock Quantisation of derived {L}agrangians.
\newblock arXiv: 1607.01000v1 [math.AG], 2016.

\bibitem[Pri5]{DQnonneg}
J.~P. Pridham.
\newblock Deformation quantisation for unshifted symplectic structures on
  derived {A}rtin stacks.
\newblock {\em Selecta Math. (N.S.)}, 24(4):3027--3059, 2018.
\newblock arXiv: 1604.04458v4 [math.AG].

\bibitem[Pri6]{DQvanish}
J.~P. Pridham.
\newblock Deformation quantisation for {$(-1)$}-shifted symplectic structures
  and vanishing cycles.
\newblock {\em Algebr. Geom.}, to appear.
\newblock arXiv:1508.07936v5 [math.AG].

\bibitem[PTVV]{PTVV}
T.~Pantev, B.~To{\"e}n, M.~Vaqui{\'e}, and G.~Vezzosi.
\newblock Shifted symplectic structures.
\newblock {\em Publ. Math. Inst. Hautes \'Etudes Sci.}, 117:271--328, 2013.
\newblock arXiv: 1111.3209v4 [math.AG].

\bibitem[Rez]{rezk}
Charles Rezk.
\newblock A model for the homotopy theory of homotopy theory.
\newblock {\em Trans. Amer. Math. Soc.}, 353(3):973--1007 (electronic), 2001.

\bibitem[TV]{hag2}
Bertrand To{\"e}n and Gabriele Vezzosi.
\newblock Homotopical algebraic geometry. {II}. {G}eometric stacks and
  applications.
\newblock {\em Mem. Amer. Math. Soc.}, 193(902):x+224, 2008.
\newblock arXiv math.AG/0404373 v7.

\end{thebibliography}

\end{document}